\documentclass[a4paper,11pt]{article}
\usepackage{stmaryrd}
\usepackage{bbm}
\usepackage[top=2.5cm, bottom=2.6cm, left=3cm, right=2.6cm]{geometry}
\usepackage{amsfonts}
\usepackage{eufrak}
\usepackage{eucal}
\usepackage{mathrsfs}
\usepackage{amsmath}
\usepackage{latexsym}
\usepackage{color}
\usepackage{amscd}
\usepackage{amssymb}
\usepackage{amsthm}
\usepackage{latexsym}
\usepackage{indentfirst}

\newtheorem{theorem}{Theorem}[section]

\newtheorem{corollary}{Corollary}[section]
\newtheorem{lemma}{Lemma}[section]
\newtheorem{proposition}{Proposition}[section]

\begin{document}

\title{Kazhdan-Lusztig bases and the asymptotic forms\\ for affine $q$-Schur algebras}
\author{Weideng Cui}
\date{}
\maketitle \abstract{We define Kazhdan-Lusztig bases and study asymptotic forms for affine $q$-Schur algebras following Du and McGerty. We will show that the analogues of Lusztig's conjectures for Hecke algebras with unequal parameters hold for affine $q$-Schur algebras. We will also show that the affine $q$-Schur algebra $\mathcal{S}_{q,k}^{\vartriangle}(2,2)$ over a field $k$ has finite global dimension when char $k=0$ and $1+q\neq 0.$}\\

 \large

\medskip
\section{Introduction}

In their influential paper [KL1] in 1979, Kazhdan and Lusztig introduced some remarkable bases for the Hecke algebra of an arbitrary Coxeter group. These bases play an important role in the representation theory of algebraic groups and quantum groups. When the Coxeter group is crystallographic (including finite and affine Weyl groups), the Kazhdan-Lusztig polynomials and the structure constants for the Kazhdan-Lusztig bases have non-negative cofficients (see [KL2] and [L2], see also [EW] for an arbitrary Coxeter group). An important application of Kazhdan-Lusztig bases and positivity properties for Hecke algebras is the asymptotic methods introduced by Lusztig in the study of cell classification of affine Weyl groups ([L2], [L3] and [L5]) and representations of Hecke algebras ([L4] and [L6]).

In [Du1], Du has introduced some new bases for the $q$-Schur algebras, which behave in a similar way as the Kazhdan-Lusztig bases for the Hecke algebras. He has also discussed some basic properties and shown that the new bases are exactly the same as the BLM's bases in [BLM]. In [Du2], Du has studied some further applications of these bases, such as the asymptotic algebras of the $q$-Schur algebras. He has established the double centralizer property and proved that the asymptotic algebra is a direct sum of full matrix rings over $\mathbb{Z}.$

In [Mc], in order to obtain the cell structure of quantum affine $\mathfrak{s}\mathfrak{l}_{n},$ McGerty first investigates the structure of cells in the affine $q$-Schur algebras $\mathcal{S}_{q}^{\vartriangle}(n,r).$ He defined distinguished elements in $\mathcal{S}_{q}^{\vartriangle}(n,r),$ and showed that all the notions of cells for $\mathcal{S}_{q}^{\vartriangle}(n,r)$ can be deduced from those for the affine Hecke algebras. He also defined the asymptotic algebra of an affine $q$-Schur algebra and obtained its structure by analogy with the affine Hecke algebra case.

In this paper, we will generalize Du's results to the affine $q$-Schur algebra case and give an explicit proof of some results of McGerty's. We will show that the analogues of Lusztig's conjectures for Hecke algebras with unequal parameters hold for affine $q$-Schur algebras. We will show that the affine $q$-Schur algebra $\mathcal{S}_{q,k}^{\vartriangle}(2,2)$ over a field $k$ has finite global dimension and its derived module category admits a stratification when char~$k=0$ and $1+q\neq 0.$ We will also show that the Kazhdan-Lusztig basis for the affine $q$-Schur algebra algebraically defined by Green is exactly the same as the canonical basis geometrically defined by Lusztig, thus we have the positivity properties for the Kazhdan-Lusztig basis, which is probably known to experts.

The organization of this article is as follows. In Section 2, we recall the definition and basic properties of Kazhdan-Lusztig bases and asymptotic forms for affine Hecke algebras of type $A.$ In Section 3, we study the Kazhdan-Lusztig bases for affine $q$-Schur algebras. In Section 4, we study the asymptotic forms and get the double centralizer property for affine $q$-Schur algebras (see Theorem 4.1 and 4.2). In Section 5, we will show that the analogues of Lusztig's conjectures for Hecke algebras with unequal parameters hold for affine $q$-Schur algebras (see Proposition 5.1). In Section 6, we will study some relatively simple examples of affine $q$-Schur algebras. In Section 7, we will show that the Kazhdan-Lusztig basis for the affine $q$-Schur algebra algebraically defined by Green is exactly the same as the canonical basis geometrically defined by Lusztig (see Theorem 7.1).

\section{Kazhdan-Lusztig bases and asymptotic forms for affine Hecke algebras}

Throughout this paper, we only consider the extended affine Weyl group $W$ of type $\tilde{A}_{r-1}.$ Recall that $W$ is the group consisting of all permutations $\sigma: \mathbb{Z}\rightarrow \mathbb{Z}$ such that $\sigma(i+r)=\sigma(i)+r$ for all $i\in \mathbb{Z}.$ Let $s_{i}$ ($0\leq i\leq r-1$) be defined by setting $s_{i}(j)=j$ for $j\not\equiv i, i+1$ mod $r,$ $s_{i}(j)=j-1$ for $j\equiv i+1$ mod $r,$ and $s_{i}(j)=j+1$ for $j\equiv i$ mod $r,$ and let $\rho$ be the permutation of $\mathbb{Z}$ taking $t$ to $t+1$ for all $t.$ Let $W'$ be the subgroup of $W$ generated by all $S=\{s_{i}\}_{0\leq i\leq r-1}$ and let $\Omega$ be the subgroup generated by $\rho,$ then it is well-known that $\Omega$ is an infinite cyclic group and $(W', S)$ is a Coxeter group. Moreover we have $W\cong \Omega \ltimes W'$ is an extended Coxeter group, the length function $l$ and the Bruhat order $\leq$ on $W'$ can be extended to $W.$\vskip2mm

2.1 Let $t=q^{\frac{1}{2}}$ be an indeterminate and let $\mathcal{A}=\mathbb{Z}[t, t^{-1}]$ be the ring of Laurent polynomials in $t.$ We set $\mathcal{A}^{-}=\mathbb{Z}[t^{-1}],$ $\mathcal{A}^{+}=\mathbb{Z}[t],$ and denote by $\mathbf{A}$ the quotient field $\mathbb{Q}(t)$ of $\mathcal{A}.$

Let $(W, S)$ be the extended Coxeter group as above. The extended affine Hecke algebra $\mathcal{H}_{\vartriangle}(r)$ over $\mathcal{A}$ corresponding to $W$ is a free $\mathcal{A}$-module with basis $\{T_{w}\}_{w\in W}$ satisfying \vskip2mm

$~~~~~~~~~~~~~~~~~~~~T_{w}T_{w'}=T_{ww'}$ ~~~~~~if ~~$l(ww')=l(w)+l(w'),$

$~~~~~~~~~~~~~~~(T_{s}+1)(T_{s}-q)=0$ ~~if ~~~~~~~~~~~$s\in S.$
\vskip2mm
The ring $\mathcal{A}$ admits an involution $-$ defined by $\overline{t}=t^{-1}.$ This extends to an involutive automorphism $-$ of $\mathcal{H}_{\vartriangle}(r),$ defined by $\overline{\sum a_{w} T_{w}}=\sum \overline{a_{w}} T_{w^{-1}}^{-1}.$ In [KL1], Kazhdan and Lusztig showed, for any $w\in W,$ there is a unique element $C_{w}'\in \mathcal{H}_{\vartriangle}(r)$ such that $$\overline{C_{w}'}=C_{w}' ~~\mathrm{and}~~ C_{w}'=\sum\limits_{y\leq w} (-1)^{l(w)-l(y)}q^{l(w)/2-l(y)} P_{y,w}(q^{-1}) T_{y},\eqno{(2.1)}$$
where $P_{y, w}$ is a polynomial in $q$ of degree $\leq \frac{1}{2}(l(w)-l(y)-1)$ for $y< w,$ and $P_{w,w}=1.$ The set $\{C_{w}'\}_{w\in W}$ forms a basis of $\mathcal{H}_{\vartriangle}(r).$

Let $j$ be the involution of the ring $\mathcal{H}_{\vartriangle}(r)$ given by $$j(\sum a_{w} T_{w})=\sum \overline{a_{w}} (-q)^{-l(w)} T_{w}.$$
Let $C_{w}=(-1)^{l(w)}j(C_{w}'),$ for each $w\in W.$ We have $$\overline{C_{w}}=C_{w} ~~\mathrm{and}~~ C_{w}=q^{-l(w)/2}\sum\limits_{y\leq w} P_{y,w}(q) T_{y}.\eqno{(2.2)}$$
$\{C_{w}'\}_{w\in W}$ and $\{C_{w}\}_{w\in W}$ are called Kazhdan-Lusztig bases. Note that our notations $C_{w}$ and $C_{w}'$ exchange these in [KL1], since we shall mainly use the elements $C_{w}.$

We define, for any $x, y, z\in W,$ some $h_{x, y, z}\in \mathcal{A}$ so that $$C_{x}C_{y}=\sum\limits_{z} h_{x,y,z} C_{z}.$$ According to [L2, (3.2.1)], $h_{x,y,z}$ has $\geq 0$ coefficients as a polynomial in $t, t^{-1},$ and $h_{x,y,z}(t)=h_{x,y,z}(t^{-1}).$\vskip2mm

2.2 For any $w\in W,$ we define $$\textbf{a}(z)= \mathrm{max}_{x,y\in W}\mathrm{deg}(h_{x,y,z}),$$
where the degree is taken with respect to $t$ and for any $x,y,z\in W,$ we define $$\gamma_{x,y,z}=\mathrm{coefficient ~of }~t^{\textbf{a}(z)}\mathrm{~in~}h_{x,y,z}.$$

In fact, $\textbf{a}(z)$ is the least non-negative integer satisfying $$t^{\textbf{a}(z)}h_{x,y,z}\in \mathcal{A}^{+},~~~~\forall~ x, y\in W,$$
and $\textbf{a}(z)=\textbf{a}(z^{-1})$ due to the fact $h_{x,y,z}=h_{y^{-1},x^{-1},z^{-1}}.$ This function on $W$ is usually called the $\textbf{a}$-function.

Let $\delta(z)$ be the degree of $P_{1,z}$ in $q$, we have (see [L3])\vskip2mm

(a) $\textbf{a}(z)\leq l(z)-2\delta(z)$ for all $z\in W.$ Put $\textbf{D}=\{z\in W|~\textbf{a}(z)=l(z)-2\delta(z)\}.$ Then it is a finite set;

(b) $\gamma_{x,y,d}\neq 0$ with $d\in \textbf{D}$ implies $x=y^{-1},$ $\gamma_{x,x^{-1},d}=1,$ and $d^{2}=1;$

(c) for each $x\in W,$ there is a unique $d\in \textbf{D}$ such that $\gamma_{x,x^{-1},d}=1;$

(d) $\gamma_{x,y,z}=\gamma_{y,z^{-1},x^{-1}}=\gamma_{z^{-1}, x, y^{-1}}=\gamma_{y^{-1},x^{-1},z^{-1}}$ for all $x,y,z\in W.$\vskip2mm

The elements in $\textbf{D}$ are called distinguished involutions.\vskip2mm

2.3 We define a preorder $\preceq_{L}$ on $W$ as follows: $x\preceq_{L}y$ if there exists $h\in \mathcal{H}_{\vartriangle}(r)$ such that the coefficient of $C_{x}$ in the linear expression of $hC_{y}$ is nonzero. The preorder $\preceq_{R}$ is defined by $x\preceq_{R}y$ if $x^{-1}\preceq_{L} y^{-1}.$ We denote by $x\preceq_{LR}y$ if there exist $h, h'\in \mathcal{H}_{\vartriangle}(r)$ such that the coefficient of $C_{x}$ in the linear expression of $hC_{y}h'$ is nonzero. Let $\sim_{L}, \sim_{R},$ and $\sim_{LR}$ be the equivalence relations defined by the preorders $\preceq_{L}, \preceq_{R},$ and $\preceq_{LR}$ respectively. The corresponding equivalence classes in $W$ are called left cells, right cells, and two-sided cells respectively.

For each $w\in W,$ let $$\mathfrak{L}(w)=\{s\in S|~sw< w\}~~~~\mathrm{and}~~~~\mathfrak{R}(w)=\{s\in S|~ws< w\}.$$
We have the following remarkable properties relating to these preorders (see [L2] and [L3]).\vskip2mm

(a) $x\preceq_{L} y\Rightarrow \mathfrak{R}(y)\subseteq \mathfrak{R}(x),$ ~ $x\preceq_{R} y\Rightarrow \mathfrak{L}(y)\subseteq \mathfrak{L}(x),$ ~$x\sim_{L} y\Rightarrow \mathfrak{R}(y)=\mathfrak{R}(x),$ ~$x\sim_{R} y\Rightarrow \mathfrak{L}(y)=\mathfrak{L}(x).$

(b) $z\preceq_{LR} z'\Rightarrow \textbf{a}(z)\geq \textbf{a}(z'),$ ~$z\sim_{LR} z'\Rightarrow \textbf{a}(z)= \textbf{a}(z').$

(c) If $\gamma_{x,y,z}\neq 0,$ then $x\sim_{L} y^{-1},$ $y\sim_{L} z,$ $x\sim_{R} z,$ and $\textbf{a}(x)=\textbf{a}(y)=\textbf{a}(z).$

(d) If $x\preceq_{L} y$ and $\textbf{a}(x)=\textbf{a}(y)$ then $x\sim_{L} y.$

(e) If $x\preceq_{L} y$ and $x\sim_{LR} y$ then $x\sim_{L} y.$

(f) For each $d\in \textbf{D}$ there is a unique left (resp. right) cell $\Gamma_{d}'$ (resp. $\Gamma_{d}$) containing $d$ and $\gamma_{y^{-1}, y, d}\neq 0$ for every $y\in \Gamma_{d}'.$\vskip2mm

2.4 Lusztig introduced in [L3] a ring $\mathcal{J}_{\mathbb{Z}}=\mathcal{J}(W)$ which is free over $\mathbb{Z}$ with a basis $\{t_{x}\}_{x\in W}$ satisfying $$t_{x}t_{y}=\sum\limits_{z\in W} \gamma_{x,y,z} t_{z},~~~~\mathrm{for}~x,y\in W.$$
Note that in the formula one have only to extend the sum on such $z\in W$ satisfying $\mathfrak{R}(z)=\mathfrak{R}(y)$ and $\mathfrak{L}(z)=\mathfrak{L}(x)$ by (a) and (c) in 2.3. It is an associative algebra with an identity element $\sum_{d\in \textbf{D}} t_{d}.$ He also proved that the $\mathcal{A}$-linear map $\phi: \mathcal{H}_{\vartriangle}(r)\rightarrow \mathcal{J}_{\mathbb{Z}}\otimes_{\mathbb{Z}} \mathcal{A},$ defined by $$\phi(C_{w})=\sum\limits_{u\in W}\sum\limits_{\substack{d\in \textbf{D}\\\textbf{a}(d)=\textbf{a}(u)}} h_{w,d,u} t_{u},~~~~w\in W,$$
is a homomorphism of $\mathcal{A}$-algebras with unit. Moreover $\phi$ is injective and becomes an isomorphism when tensored with $\textbf{A}.$ Moreover the equivalence relations $\sim_{L}$ and $\sim_{LR}$ can be described as follows: for any $y, w\in W$

(a) $y\sim_{L} w\Leftrightarrow t_{y}t_{w^{-1}}\neq 0,$

(b) $y\sim_{LR} w\Leftrightarrow t_{y}t_{x}t_{w}\neq 0$ for some $x\in W.$\vskip2mm

2.5 Let $S_{0}=\{s_{i}\}_{1\leq i\leq r-1}$ and let $I, J$ be subsets of $S_{0}.$ We denote by $\mathcal{D}_{IJ}$ (resp. $\mathcal{D}_{IJ}^{+}$) the set of distinguished representatives of minimal (resp. maximal) length in the double cosets $W_{I}\backslash W/W_{J},$ where $W_{K}$ is the parabolic subgroup of $W$ generated by $K\subseteq S_{0}.$ Each double coset $D\in W_{I}\backslash W/W_{J}$ is characterized as the set $$D=\{y\in W|~x\leq y\leq x^{+}\}$$
where $x\in D\cap \mathcal{D}_{IJ}$ and $x^{+}\in D\cap \mathcal{D}_{IJ}^{+}.$ We denote the obvious bijection from $\mathcal{D}_{IJ}$ to $\mathcal{D}_{IJ}^{+}$ by $x\mapsto x^{+},$ and let $\mathcal{D}_{I}=\mathcal{D}_{I\emptyset}.$

We define an $\mathcal{A}$-submodule $\mathcal{H}_{IJ}$ of $\mathcal{H}_{\vartriangle}(r)$ which is spanned by the standard basis elements $$T_{D}=\sum\limits_{x\in D} T_{x},~~~~D\in W_{I}\backslash W/W_{J},$$
then we have the following lemma by analogy with the finite Weyl group case (see [C, (1.9) and (1.10)])
\begin{lemma}
Maintain the above notation.

$(1)$ The $A$-submodule $\mathcal{H}_{IJ}$ of $\mathcal{H}_{\vartriangle}(r)$ is characterized as the set $$\{h\in \mathcal{H}_{\vartriangle}(r)|~T_{s}h=qh=hT_{s'}, ~\forall s\in I, s'\in J\}.$$

$(2)$ The elements $\{C_{w}|~w\in \mathcal{D}_{IJ}^{+}\}$ form an $\mathcal{A}$-basis of $\mathcal{H}_{IJ}.$

$(3)$ $\overline{\mathcal{H}_{IJ}}=\mathcal{H}_{IJ}.$
\end{lemma}

Let $\Psi: \mathcal{H}_{\vartriangle}(r)\rightarrow \mathcal{H}_{\vartriangle}(r)$ be the ring homomorphism defined by $\Psi(t)=-t, \Psi(T_{x})=(-q)^{l(x)}T_{x^{-1}}^{-1}, x\in W.$ Then $\Psi^{2}=1$ and $C_{x}'=\Psi(C_{x}).$ Applying $\Psi$ to $\mathcal{H}_{IJ},$ we have the following lemma by analogy with the finite Weyl group case (see [Du1, Lemma 1.7]).
\begin{lemma}
Let $\widetilde{\mathcal{H}}_{IJ}=\Psi(\mathcal{H}_{IJ}).$ Then we have

$(1)$ $\widetilde{\mathcal{H}}_{IJ}$ can be characterized as the set $$\{h\in \mathcal{H}_{\vartriangle}(r)|~T_{s}h=-h=hT_{s'}, ~\forall s\in I, s'\in J\}.$$

$(2)$ The elements $\{C_{x}'|~x\in \mathcal{D}_{IJ}^{+}\}$ form a basis of $\widetilde{\mathcal{H}}_{IJ}.$

$(3)$ The elements $$\widetilde{T}_{D}=\sum\limits_{w\in D} (-q)^{-l(w)} T_{w},~~D\in W_{I}\backslash W/W_{J}$$
also form a basis of $\widetilde{\mathcal{H}}_{IJ}.$
\end{lemma}

We remark that $x\in \mathcal{D}_{IK}^{+}$ if and only if $I\subseteq \mathfrak{L}(x)$ and $K\subseteq \mathfrak{R}(x).$ We have the following lemma by analogy with finte Weyl group case (see [Du2, Lemma 1.4.1]).
\begin{lemma}
Let $\mathcal{J}_{IK}$ be the additive subgroup of $\mathcal{J}_{\mathbb{Z}}$ generated by basis elements $\{t_{w}|~w\in \mathcal{D}_{IK}^{+}\}.$ Then we have

$(a)$ $\mathcal{J}_{KL}\mathcal{J}_{IJ}\subseteq \mathcal{J}_{KJ}$ for any subsets $I, J, K, L$ of $S_{0}.$ In particular, $\mathcal{J}_{KK}$ is a subring of $\mathcal{J}_{\mathbb{Z}}$ with the identity element $\sum_{d\in \textbf{D}\cap \mathcal{D}_{KK}^{+}} t_{d}.$

$(b)$ $\mathcal{J}_{IK}$ is a $\mathcal{J}_{II}$-$\mathcal{J}_{KK}$-bimodule.

$(c)$ $\mathrm{Hom}$$_{\mathcal{J}_{\mathbb{Z}}}(\mathcal{J}_{K\emptyset}, \mathcal{J}_{I\emptyset})$ is free of rank $|\mathcal{D}_{IK}|.$
\end{lemma}
\begin{proof}
The statements (a) and (b) follows from 2.3 (a), (c) and the previous remark on $\mathcal{D}_{IK}^{+}.$

(c) Note that $\mathcal{J}_{\mathbb{Z}}=\mathcal{J}_{\emptyset\emptyset}$ and $\mathcal{D}_{K}^{+}$ is a union of right cells. So if $t_{K}=\sum_{d\in \textbf{D}\cap \mathcal{D}_{K}^{+}}t_{d},$ then we have $$\mathcal{J}_{K\emptyset}=t_{K}\mathcal{J}_{\mathbb{Z}}=\sum_{d\in \textbf{D}\cap \mathcal{D}_{K}^{+}}t_{d}\mathcal{J}_{\mathbb{Z}},$$
where $t_{d}\mathcal{J}_{\mathbb{Z}}$ is spanned by the elements $t_{x},$ $x\in \Gamma_{d}.$

For each $w\in \mathcal{D}_{IK}^{+},$ by 2.2 (c), there is a unique element $d_{w}\in \textbf{D}$ such that $\gamma_{w^{-1}, w, d_{w}}=1.$ By 2.3 (a) and (c), we have $d_{w}\in \mathcal{D}_{KK}^{+}.$ We define a $\mathcal{J}_{\mathbb{Z}}$-homomorphism $\varphi_{w}$: $\mathcal{J}_{K\emptyset}\rightarrow\mathcal{J}_{I\emptyset}$ by setting
$$\varphi_{w}(t_{x})=\begin{cases}t_{w}t_{x}& \hbox {if } t_{d_{w}}t_{x}\neq 0; \\~0& \hbox {otherwise}.\end{cases}$$
It is easy to see that the elements $\varphi_{w},$ $w\in \mathcal{D}_{IK}^{+}$ are linearly independent and form a basis for $\mathrm{Hom}$$_{\mathcal{J}_{\mathbb{Z}}}(\mathcal{J}_{K\emptyset}, \mathcal{J}_{I\emptyset}).$\end{proof}

\section{Kazhdan-Lusztig bases for affine $q$-Schur algebras}

3.1 Let $\mathbb{Z}_{\vartriangle}^{n}=\{(\lambda_{i})_{i\in \mathbb{Z}}|~\lambda_{i}\in \mathbb{Z}, \lambda_{i}=\lambda_{i-n}~\mathrm{for}~i\in \mathbb{Z}\}$ and $\mathbb{N}_{\vartriangle}^{n}=\{(\lambda_{i})_{i\in \mathbb{Z}}\in \mathbb{Z}_{\vartriangle}^{n}|~\lambda_{i}\geq 0~\mathrm{for}~i\in \mathbb{Z}\}.$ For $\lambda=(\lambda_{i})_{i\in \mathbb{Z}}\in \mathbb{Z}_{\vartriangle}^{n},$ let $\sigma(\lambda)=\sum\limits_{1\leq i\leq n}\lambda_{i}.$ For $r\geq 0,$ we set $\Lambda_{\vartriangle}(n,r)=\{\lambda\in \mathbb{N}_{\vartriangle}^{n}|~\sigma(\lambda)=r\}.$ For each $\lambda\in \Lambda_{\vartriangle}(n,r),$ it determines naturally a subset $I(\lambda)$ of $S_{0}.$ We will use the standard notations $W_{\lambda}, \mathcal{D}_{\lambda}^{\vartriangle}, \mathcal{D}_{\lambda\mu}^{\vartriangle},\ldots$ instead of $W_{I(\lambda)}, \mathcal{D}_{I(\lambda)}, \mathcal{D}_{I(\lambda)I(\mu)},\ldots$ for $W,$ thus $W_{\lambda}: =\mathfrak{S}_{(\lambda_1,\ldots,\lambda_{n})}$ is the standard Young subgroup of $\mathfrak{S}_{r}.$

For each $\lambda\in \Lambda_{\vartriangle}(n,r),$ let $x_{\lambda}=\sum_{w\in W_{\lambda}}T_{w}\in \mathcal{H}_{\vartriangle}(r)$ and $y_{\lambda}=j(x_{\lambda}),$ then the right ideal $x_{\lambda}\mathcal{H}_{\vartriangle}(r)$ of $\mathcal{H}_{\vartriangle}(r)$ has a basis $\{x_{\lambda}T_{z}|~z\in \mathcal{D}_{\lambda}^{\vartriangle}\}.$ The endomorphism algebra $$\mathcal{S}_{q}^{\vartriangle}(n,r): =\mathrm{End}_{\mathcal{H}_{\vartriangle}(r)}\Big(\bigoplus_{\lambda\in \Lambda_{\vartriangle}(n,r)} x_{\lambda}\mathcal{H}_{\vartriangle}(r)\Big)$$
is called an affine $q$-Schur algebra (see [Gr]). If we specialize $t$ to 1, then $\mathcal{S}_{1}^{\vartriangle}(n,r)$ is the affine Schur algebra over $\mathbb{Z}.$

For $\lambda, \mu\in \Lambda_{\vartriangle}(n,r)$ and $w\in \mathcal{D}_{\lambda\mu}^{\vartriangle},$ define $\phi_{\lambda,\mu}^{w}\in \mathcal{S}_{q}^{\vartriangle}(n,r)$ by $$\phi_{\lambda,\mu}^{w}(x_{\nu}h)=\delta_{\mu\nu}\sum\limits_{x\in W_{\lambda}wW_{\mu}}T_{x}h,\eqno{(3.1)}$$
where $\nu\in \Lambda_{\vartriangle}(n,r)$ and $h\in \mathcal{H}_{\vartriangle}(r).$ It has been proved in [Gr, Theorem 2.2.4] that the set $\{\phi_{\lambda,\mu}^{w}|~\lambda, \mu\in \Lambda_{\vartriangle}(n,r), w\in \mathcal{D}_{\lambda\mu}^{\vartriangle}\}$ forms a basis of $\mathcal{S}_{q}^{\vartriangle}(n,r).$ This is called the standard basis of $\mathcal{S}_{q}^{\vartriangle}(n,r).$ The set of basis elements $\{\phi_{\lambda,\mu}^{w}|~\lambda, \mu\in \Lambda_{\vartriangle}(n,r), w\in \mathfrak{S}_{r}\cap \mathcal{D}_{\lambda\mu}^{\vartriangle}\}$ spans a subalgebra of $\mathcal{S}_{q}^{\vartriangle}(n,r)$ canonically isomorphic to the $q$-Schur algebra $\mathcal{S}_{q}(n,r).$ If $n\geq r,$ let $\omega=(\ldots, 1^{r}, 0^{n-r},\ldots)\in \Lambda_{\vartriangle}(n,r),$ The set of basis elements $\{\phi_{\omega,\omega}^{w}|~ w\in W\}$ spans a subalgebra of $\mathcal{S}_{q}^{\vartriangle}(n,r)$ canonically isomorphic to the affine Hecke algebra $\mathcal{H}_{\vartriangle}(r),$ where $\phi_{\omega,\omega}^{w}$ is identified with $T_{w}$ (see [Gr, Proposition 2.2.5]).\vskip2mm

3.2 For a positive integer $n,$ let $\Theta_{\vartriangle}(n)$ be the set of all matrices $A=(a_{i,j})_{i,j\in \mathbb{Z}}$ with $a_{i,j}\in \mathbb{N}$ such that

(a) $a_{i,j}=a_{i+n,j+n}$ for $i,j\in \mathbb{Z};$

(b) for every $i\in \mathbb{Z},$ both sets $\{j\in \mathbb{Z}|~a_{i,j}\neq 0\}$ and $\{j\in \mathbb{Z}|~a_{j,i}\neq 0\}$ are finite.

For $A\in \Theta_{\vartriangle}(n)$ and $r\geq 0,$ let $\sigma(A)=\sum\limits_{1\leq i\leq n,~j\in \mathbb{Z}}a_{i,j}$ and $\Theta_{\vartriangle}(n,r)=\{A\in \Theta_{\vartriangle}(n)|~\sigma(A)=r\}.$

By [VV, 7.4] (see also [DF1]), there is a bijective map $j_{\vartriangle}: \mathcal{B}=\{(\lambda, w, \mu)|~$\\$\lambda, \mu\in \Lambda_{\vartriangle}(n,r), w\in \mathcal{D}_{\lambda\mu}^{\vartriangle}\}\rightarrow \Theta_{\vartriangle}(n,r)$ sending $(\lambda, w, \mu)$ to the matrix $A=(|R_{k}^{\lambda}\cap wR_{l}^{\mu}|)_{k,l\in \mathbb{Z}},$ where $R_{i+kn}^{\nu}=\{\nu_{k,i-1}+1, \nu_{k,i-1}+2,\ldots, \nu_{k,i-1}+\nu_{i}=\nu_{k,i}\}$ with $\nu_{k,i-1}=kr+\sum\limits_{1\leq t\leq i-1}\nu_{t}$ for all $1\leq i\leq n, k\in \mathbb{Z}$ and $\nu\in \Lambda_{\vartriangle}(n,r).$ Thus we have also obtained a bijective map $j_{\vartriangle}': \mathcal{B}'=\{(\lambda, w^{+}, \mu)|~\lambda, \mu\in \Lambda_{\vartriangle}(n,r), w^{+}\in \mathcal{D}_{\lambda\mu}^{\vartriangle,+}\}\rightarrow \Theta_{\vartriangle}(n,r),$ which has been given in [Mc]. We denote their inverses by $\sigma': \Theta_{\vartriangle}(n,r)\rightarrow \mathcal{B},$ $A\mapsto (\lambda, w_{A}, \mu)$ and $\sigma: \Theta_{\vartriangle}(n,r)\rightarrow \mathcal{B}',$ $A\mapsto (\lambda, w_{A}^{+}, \mu),$ where $\lambda=ro(A)=(\sum a_{1,i},\ldots, \sum a_{n,i})$ and $\mu=co(A)=(\sum a_{i,1},\ldots, \sum a_{i,n}).$ We will not distinguish between them. We will also denote by $\sigma(A)=w_{A}^{+}$ if $A$ corresponds to $(\lambda, w_{A}^{+}, \mu)$ in the following and hope that it doesn't cause any confusion.

Recall from 2.5 that for $w\in \mathcal{D}_{\lambda\mu}^{\vartriangle},$ $C_{w^{+}}\in \mathcal{H}_{\lambda\mu}.$ Write $$C_{w^{+}}=\sum\limits_{z\in \mathcal{D}_{\lambda\mu}^{\vartriangle}} \alpha_{z,w} T_{W_{\lambda}zW_{\mu}}.\eqno{(3.2)}$$

The following lemma gives the explicit expression of $\alpha_{z,w},$ which implies that the sum occurring above is in fact a finite sum.
\begin{lemma}
For each $w\in \mathcal{D}_{\lambda\mu}^{\vartriangle},$ if $$C_{w^{+}}=\sum\limits_{z\in \mathcal{D}_{\lambda\mu}^{\vartriangle}} \alpha_{z,w} T_{W_{\lambda}zW_{\mu}},$$
then $z^{+}\leq w^{+}$ and $\alpha_{z, w}=t^{-l(w^{+})}P_{z^{+}, w^{+}}.$ Therefore the degree in $t$ of $\alpha_{z,w}$ is at most $-l(z^{+})-1.$
\end{lemma}
\begin{proof}This follows from the expression of $C_{w^{+}}$ in (2.2) and Lemma 2.1.\end{proof}

We define for $B=(\lambda, w_{B}, \mu)\in \mathcal{B},$$$\theta_{B}=t^{l(w_{0,\mu})}\sum\limits_{z\in \mathcal{D}_{\lambda\mu}^{\vartriangle}} \alpha_{z,w_{B}} \phi_{\lambda,\mu}^{z},\eqno{(3.3)}$$
where $w_{0,\mu}$ is the longest element of $W_{\mu}.$

The following theorem can be proved by analogy with the $q$-Schur algebra case (see [Du1, Theorem 2.3]).
\begin{theorem}
$(1)$ The elements $\{\theta_{B}\}_{B\in \mathcal{B}}$ form a basis of $\mathcal{S}_{q}^{\vartriangle}(n,r).$

$(2)$ If $B=(\lambda,w_{B},\mu)\in \mathcal{B}$ and regarding $\theta_{B}$ as a map from $x_{\mu}\mathcal{H}_{\vartriangle}(r)$ to $x_{\lambda}\mathcal{H}_{\vartriangle}(r)$, then we have $\theta_{B}(C_{w_{0,\mu}})=C_{w_{B}^{+}}.$
\end{theorem}
\begin{proof}
$(1)$ Let $\phi_{\lambda}=\phi_{\lambda,\lambda}^{1}.$ Then $\phi_{\lambda}$ is an idempotent of $\mathcal{S}_{q}^{\vartriangle}(n,r)$ and $$\mathcal{S}_{q}^{\vartriangle}(n,r)=\bigoplus_{\lambda, \mu\in \Lambda_{\vartriangle}(n,r)}\phi_{\lambda}\mathcal{S}_{q}^{\vartriangle}(n,r)\phi_{\mu}.$$
For $\lambda, \mu\in \Lambda_{\vartriangle}(n,r),$ it is obvious that we have $$\phi_{\lambda}\mathcal{S}_{q}^{\vartriangle}(n,r)\phi_{\mu}\cong \mathrm{Hom}_{\mathcal{H}_{\vartriangle}(r)}(x_{\mu}\mathcal{H}_{\vartriangle}(r), x_{\lambda}\mathcal{H}_{\vartriangle}(r)),$$
and $\phi_{\lambda,\mu}^{w}, w\in \mathcal{D}_{\lambda\mu}^{\vartriangle}$ form a basis of $\phi_{\lambda}\mathcal{S}_{q}^{\vartriangle}(n,r)\phi_{\mu}.$ Now the map $$\mathrm{Hom}_{\mathcal{H}_{\vartriangle}(r)}(x_{\mu}\mathcal{H}_{\vartriangle}(r), x_{\lambda}\mathcal{H}_{\vartriangle}(r))\rightarrow \mathcal{H}_{\lambda\mu}$$ given by sending $\phi_{\lambda,\mu}^{w}$ to $\phi_{\lambda,\mu}^{w}(x_{\mu})$ is an $\mathcal{A}$-module isomorphism by 2.5 and (3.1). Therefore, we obtain that the elements $$\sum\limits_{z\in \mathcal{D}_{\lambda\mu}^{\vartriangle}} \alpha_{z,w} \phi_{\lambda,\mu}^{z},~~w\in \mathcal{D}_{\lambda\mu}^{\vartriangle}$$ form a basis of $\mathrm{Hom}_{\mathcal{H}_{\vartriangle}(r)}(x_{\mu}\mathcal{H}_{\vartriangle}(r), x_{\lambda}\mathcal{H}_{\vartriangle}(r))$ by Lemma 2.1 and (3.2). Therefore, $\{\theta_{B}\}_{B\in \mathcal{B}}$ is a basis of $\mathcal{S}_{q}^{\vartriangle}(n,r).$

(2) Note that from (2.2) we have $C_{w_{0,\mu}}=t^{-l(w_{0,\mu})}x_{\mu}.$ Thus we have\vskip1mm
$\theta_{B}(C_{w_{0,\mu}})=t^{-l(w_{0,\mu})}\theta_{B}(x_{\mu})=\sum\limits_{z\in \mathcal{D}_{\lambda\mu}^{\vartriangle}} \alpha_{z,w_{B}} \phi_{\lambda,\mu}^{z}(x_{\mu})=C_{w_{B}^{+}}.$
\end{proof}

3.3 Since $\Psi(q)=q$ and $\Psi(T_{x})=(-q)^{l(x)}T_{x^{-1}}^{-1}, x\in W,$ $\Psi$ is a $\mathbb{Z}[q,q^{-1}]$-algebra automorphism of $\mathcal{H}_{\vartriangle}(r).$ Thus, by analogy with the $q$-Schur algebra case, $\Psi$ induces an algebra isomorphism from $\mathcal{S}_{q}^{\vartriangle}(n,r)$ to the algebra $$\widetilde{\mathcal{S}}_{q}^{\vartriangle}(n,r): =\mathrm{End}_{\mathcal{H}_{\vartriangle}(r)}\Big(\bigoplus_{\lambda\in \Lambda_{\vartriangle}(n,r)} y_{\lambda}\mathcal{H}_{\vartriangle}(r)\Big).$$
Moreover, $\widetilde{\mathcal{S}}_{q}^{\vartriangle}(n,r)$ had standard basis elements $\psi_{\lambda,\mu}^{w}, \forall \lambda, \mu\in \Lambda_{\vartriangle}(n,r), w\in \mathcal{D}_{\lambda\mu}^{\vartriangle},$ where $$\psi_{\lambda,\mu}^{w}(y_{\nu}h)=\delta_{\mu\nu}\sum\limits_{x\in W_{\lambda}wW_{\mu}}(-q)^{-l(x)}T_{x}h,$$
for all $h\in \mathcal{H}_{\vartriangle}(r),$ $\nu\in \Lambda_{\vartriangle}(n,r).$

As in (3.3), we define for each $B=(\lambda, w_{B}, \mu)\in \mathcal{B},$$$\widetilde{\theta}_{B}=(-t)^{-l(w_{0,\mu})}\sum\limits_{z\in \mathcal{D}_{\lambda\mu}^{\vartriangle}} \widetilde{\alpha}_{z,w_{B}} \psi_{\lambda,\mu}^{z},$$
where $\widetilde{\alpha}_{z,w_{B}}$ is defined by $$C_{w_{B}^{+}}'=\sum\limits_{z\in \mathcal{D}_{\lambda\mu}^{\vartriangle}} \widetilde{\alpha}_{z,w_{B}} \widetilde{T}_{W_{\lambda}zW_{\mu}}.$$
This definition makes sense by (2.1) and Lemma 2.2 (3).

By analogy with the $q$-Schur algebra case (see [Du1, Theorem 2.6]), we have the following theorem.
\begin{theorem}
Maintain the above notation.

$(1)$ The elements $\{\widetilde{\theta}_{B}\}_{B\in \mathcal{B}}$ form a basis of $\widetilde{\mathcal{S}}_{q}^{\vartriangle}(n,r).$

$(2)$ For each $B=(\lambda, w_{B}, \mu)\in \mathcal{B},$ $\widetilde{\theta}_{B}(C_{w_{0,\mu}}')=C_{w_{B}^{+}}'.$

$(3)$ $\Psi(\theta_{B})=\widetilde{\theta}_{B}, \forall B\in \mathcal{B}.$
\end{theorem}
\begin{proof}
The proof of (1) and (2) is similar to that of Theorem 3.1. We now prove (3).

Since $C_{w_{0,\mu}}'=(-t)^{l(w_{0,\mu})}y_{\mu}$ by (2.1). We have $$\Psi(x_{\mu})=\Psi(t^{l(w_{0,\mu})}C_{w_{0,\mu}})=(-t)^{l(w_{0,\mu})}C_{w_{0,\mu}}'=q^{l(w_{0,\mu})}y_{\mu}.$$
On the other hand, if $\phi_{\lambda,\mu}^{w}(x_{\mu})=h_{w}x_{\mu}$ for some $h_{w}\in \mathcal{H}_{\vartriangle}(r),$ then by definition we have $\Psi(\phi_{\lambda,\mu}^{w})(y_{\mu})=\Psi(h_{w})y_{\mu}.$ Thus, for $B=(\lambda, w_{B}, \mu)\in \mathcal{B},$ we have
\begin{align*}
\Psi(\theta_{B})(C_{w_{0,\mu}}')&=(-t)^{l(w_{0,\mu})}\sum\limits_{z\in \mathcal{D}_{\lambda\mu}^{\vartriangle}} \Psi(\alpha_{z,w_{B}}) \Psi(\phi_{\lambda,\mu}^{z})(C_{w_{0,\mu}}')\\
&=(-t)^{l(w_{0,\mu})}\sum\limits_{z\in \mathcal{D}_{\lambda\mu}^{\vartriangle}} \Psi(\alpha_{z,w_{B}}) \Psi(h_{z})C_{w_{0,\mu}}'\\
&=(-t)^{2l(w_{0,\mu})}\sum\limits_{z\in \mathcal{D}_{\lambda\mu}^{\vartriangle}} \Psi(\alpha_{z,w_{B}}) \Psi(h_{z})y_{\mu}\end{align*}
\begin{align*}
&=\sum\limits_{z\in \mathcal{D}_{\lambda\mu}^{\vartriangle}} \Psi(\alpha_{z,w_{B}}) \Psi(h_{z})\Psi(x_{\mu})\\
&=\Psi(C_{w_{B}^{+}})=C_{w_{B}^{+}}'=\widetilde{\theta}_{B}(C_{w_{0,\mu}}').
\end{align*}
Therefore, we have $\Psi(\theta_{B})=\widetilde{\theta}_{B}.$
\end{proof}

We call $\{\theta_{B}\}_{B\in \mathcal{B}}$ and $\{\widetilde{\theta}_{B}\}_{B\in \mathcal{B}}$ the Kazhdan-Lusztig basis of the affine $q$-Schur algebras.\vskip2mm

3.4 For simplicity, we denote $\phi_{\lambda,\mu}^{w_{B}}$ by $\phi_{B},$ where $B=(\lambda, w_{B}, \mu)\in \mathcal{B}.$ We first extend the involution $-$ on $\mathcal{H}_{\vartriangle}(r)$ to the affine $q$-Schur algebra $\mathcal{S}_{q}^{\vartriangle}(n,r)$ following [Du1, DF2]. For any $B=(\lambda, w_{B}, \mu)\in \mathcal{B},$ we have $\phi_{B}(C_{w_{0,\mu}})\in \mathcal{H}_{\lambda\mu},$ and hence $\overline{\phi_{B}(C_{w_{0,\mu}})}\in \mathcal{H}_{\lambda\mu}$ by Lemma 2.1 (3). We now define a map $-: \mathcal{S}_{q}^{\vartriangle}(n,r)\rightarrow \mathcal{S}_{q}^{\vartriangle}(n,r)$ such that $$\overline{t}=t^{-1},~~~~\overline{\phi_{B}}(C_{w_{0,\mu}}h)=\overline{\phi_{B}(C_{w_{0,\mu}})}h~~\forall h\in \mathcal{H}_{\vartriangle}(r).$$

By analogy with the $q$-Schur algebra case, we have the following proposition.
\begin{proposition}
$(1)$ The map $-$ is an algebra involution.

$(2)$ If $n\geq r,$ then the restriction of $-$ on $\mathcal{H}_{\vartriangle}(r)$ as a subalgebra of $\mathcal{S}_{q}^{\vartriangle}(n,r)$ coincides with the involution $-$ on $\mathcal{H}_{\vartriangle}(r).$

$(3)$ $\overline{\theta_{B}}=\theta_{B}$ for all $B\in \mathcal{B}.$
\end{proposition}

Recall that for each $B=(\lambda, w_{B}, \mu)\in \mathcal{B},$ $B$ uniquely determines the element $\sigma(B)=w_{B}^{+}\in \mathcal{D}_{\lambda\mu}^{\vartriangle,+}.$ For any $A,B\in \mathcal{B},$ we define elements $g_{A,B,C}\in \mathcal{A}$ by $$\theta_{A}\theta_{B}=\sum\limits_{C\in \mathcal{B}}g_{A,B,C} \theta_{C}.$$
The following result shows that the structure constants for $\mathcal{S}_{q}^{\vartriangle}(n,r)$ are determined by that for $\mathcal{H}_{\vartriangle}(r)$ (see also [Mc, Lemma 2.2]).
\begin{proposition}
For any $A, B, C\in \mathcal{B},$ we have $g_{A,B,C}\neq 0$ only if $co(A)=ro(B)$ and $(ro(A), co(B))=(ro(C), co(C)).$ In this case, there exists a Laurent polynomial $h_{A,B}\in \mathcal{A}$ such that $h_{A,B}\cdot g_{A,B,C}=h_{\sigma(A),\sigma(B),\sigma(C)}.$
\end{proposition}
\begin{proof}
Let $A=(\lambda, w_{A}, \mu),$ $B=(\nu,w_{B},\rho).$ The first claim follows from the definition. We now assume $\mu=\nu.$ According to (2.2) and Lemma 2.1 (1), we have $$C_{w_{0,\mu}}C_{w_{B}^{+}}=h_{\mu}C_{w_{B}^{+}},~~~\mathrm{where}~~h_{\mu}=t^{-l(w_{0,\mu})}\sum\limits_{w\in W_{\mu}}t^{2l(w)}.$$ Now we have
\begin{align*}
\sum\limits_{C\in \mathcal{B}}g_{A,B,C} C_{\sigma(C)}&=\theta_{A}\theta_{B}(C_{w_{0, \rho}})=\theta_{A}(C_{\sigma(B)})=\theta_{A}(h_{\mu}^{-1}C_{w_{0,\mu}}C_{\sigma(B)})\end{align*}
\begin{align*}
&=h_{\mu}^{-1}\theta_{A}(C_{w_{0,\mu}})C_{\sigma(B)}=h_{\mu}^{-1}C_{\sigma(A)}C_{\sigma(B)}\\
&=\sum\limits_{C\in \mathcal{B}}h_{\mu}^{-1}h_{\sigma(A),\sigma(B),\sigma(C)}C_{\sigma(C)}.
\end{align*}
Therefore, let $h_{A,B}=h_{\mu},$ then we have $h_{A,B}\cdot g_{A,B,C}=h_{\sigma(A),\sigma(B),\sigma(C)}$ for all $C\in \mathcal{B}.$
\end{proof}

For each $B=(\lambda, w_{B}, \mu)\in \mathcal{B},$ we define $$\widehat{\phi}_{B}=t^{-l(w_{B}^{+})+l(w_{0,\mu})}\phi_{B}.$$

By Lemma 3.1 and (3.3), we have
\begin{align*}
\hspace*{3cm}\theta_{B}&=t^{l(w_{0,\mu})}\sum\limits_{z\in \mathcal{D}_{\lambda\mu}^{\vartriangle},~z^{+}\leq w^{+}} t^{-l(w_{B}^{+})}P_{z^{+}, w^{+}}\phi_{(\lambda,z,\mu)}~~~~~~~~~~~~(3.4)\\
\hspace*{3cm}&=\sum\limits_{z\in \mathcal{D}_{\lambda\mu}^{\vartriangle},~z^{+}\leq w^{+}}t^{-l(w_{B}^{+})+l(z^{+})}P_{z^{+}, w^{+}}\widehat{\phi}_{(\lambda,z,\mu)}. ~~~~~~~~~~~~~(3.5)
\end{align*}

Therefore, we get $$\theta_{B}\in \widehat{\phi}_{B}+t^{-1}\sum\limits_{z\in \mathcal{D}_{\lambda\mu}^{\vartriangle},~z^{+}\leq w^{+}}\mathcal{A}^{-}\widehat{\phi}_{(\lambda,z,\mu)}.$$

Thus, we have $\{\theta_{B}\}_{B\in \mathcal{B}}$ is an IC basis (see [Du1, Definition 3.7]). Note that (3.4) has been given in [Gr, Definition 2.4.3].

\section{Asymptotic forms for affine $q$-Schur algebras}

4.1 Recall that in Proposition 3.2 we have established the relationships between the structure constants for $\mathcal{S}_{q}^{\vartriangle}(n,r)$ and $\mathcal{H}_{\vartriangle}(r).$ We extend the $\textbf{a}$-function on $\mathcal{H}_{\vartriangle}(r)$ to $\Theta_{\vartriangle}(n,r)$ by letting $\textbf{a}(A)$: $=\textbf{a}(\sigma(A))$ for all $A\in \Theta_{\vartriangle}(n,r),$ and extend the finite set $\textbf{D}$ of distinguished involutions in $W$ to the set $\textbf{D}_{\vartriangle}(n,r)=\{A\in \Theta_{\vartriangle}(n,r)|~ro(A)=co(A) ~\mathrm{and}~ \sigma(A)\in \textbf{D}\}$ (see [Mc, Lemma 3.7]). Note that the $\textbf{a}$-function defined here is different from that in [Mc, Proposition 3.8 (4)]. We denote by $\Theta_{\vartriangle}(n,r)_{\lambda\mu}=\{B\in \Theta_{\vartriangle}(n,r)|~ro(B)=\lambda, co(B)=\mu\}.$

Put $$\gamma_{A,B,C}=\begin{cases}\gamma_{\sigma(A),\sigma(B),\sigma(C)}& \hbox {if } g_{A,B,C}\neq 0; \\~~~~~~~ 0& \hbox {otherwise},\end{cases}$$
and let $\mathcal{J}_{\vartriangle}(n,r)$ be a free abelian group with basis $\{t_{A}|~A\in \Theta_{\vartriangle}(n,r)\}.$ We define a multiplication on $\mathcal{J}_{\vartriangle}(n,r)$ by setting (see [Mc, (4.7)])$$t_{A}t_{B}=\sum\limits_{C}\gamma_{A,B,C}t_{C}.$$

Using the associativity of the algebra $\mathcal{J}_{\mathbb{Z}}$ and Lemma 2.3 (a), we obtain the following lemma.
\begin{lemma}
The $\mathbb{Z}$-algebra $\mathcal{J}_{\vartriangle}(n,r)$ is associative with the identity element $$\sum\limits_{\lambda\in \Lambda_{\vartriangle}(n,r)}\sum\limits_{D\in \textbf{D}_{\vartriangle}(n,r)_{\lambda}}t_{D},$$
where $\textbf{D}_{\vartriangle}(n,r)_{\lambda}=\textbf{D}_{\vartriangle}(n,r)\cap \Theta_{\vartriangle}(n,r)_{\lambda\lambda}.$
\end{lemma}

\begin{lemma}
For $A=(\lambda,1,\mu),$ $B=(\mu,w_{B},\nu)\in \Theta_{\vartriangle}(n,r)$ with $W_{\lambda}\subset W_{\mu},$ we define $C=(\lambda,w',\nu),$ where $w_{B}^{+}\in W_{\lambda}w'W_{\nu},$ then $\theta_{A}\theta_{B}=\theta_{C}.$
\end{lemma}
\begin{proof}
It follows from the following identity:\vskip2mm $\hspace*{3cm}C_{w_{0,\mu}}C_{w_{B}^{+}}=t^{-l(w_{0,\mu})}\sum\limits_{w\in W_{\mu}}t^{2l(w)}C_{w_{B}^{+}}.$
\end{proof}

Analogous to the affine Hecke algebra case, we have the following theorem (see also [Mc, (4.9)]).
\begin{theorem}
The $\mathcal{A}$-module homomorphism $\Phi: \mathcal{S}_{q}^{\vartriangle}(n,r)\rightarrow \mathcal{J}_{\vartriangle}(n,r)\otimes_{\mathbb{Z}} \mathcal{A}$ defined by $$\Phi(\theta_{A})=\sum\limits_{B\in \Theta_{\vartriangle}(n,r)}\sum\limits_{\substack{D\in \textbf{D}_{\vartriangle}(n,r)_{\mu}\\ \textbf{a}(D)=\textbf{a}(B)}} g_{A,D,B} t_{B},~~\mathrm{where}~\mu=co(A)$$
is an algebra homomorphism and becomes an isomorphism when tensored with $\mathbf{A}.$
\end{theorem}
\begin{proof}
For any $\lambda,\mu,\nu\in \Lambda_{\vartriangle}(n,r)$ and $A\in \Theta_{\vartriangle}(n,r)_{\lambda\mu},$ $B\in \Theta_{\vartriangle}(n,r)_{\mu\nu},$ we want to prove $$\Phi(\theta_{A}\theta_{B})=\Phi(\theta_{A})\Phi(\theta_{B}).$$
This comes down to proving that $$\sum\limits_{\substack{C\in \Theta_{\vartriangle}(n,r)_{\lambda\nu}\\D\in \textbf{D}_{\vartriangle}(n,r)_{\nu}\\\textbf{a}(D)=\textbf{a}(E)}} g_{A,B,C}\cdot g_{C,D,E}=\sum\limits_{\substack{U,V\in \Theta_{\vartriangle}(n,r)\\D'\in \textbf{D}_{\vartriangle}(n,r)_{\mu}\\D''\in \textbf{D}_{\vartriangle}(n,r)_{\nu}\\\textbf{a}(D')=\textbf{a}(U)\\\textbf{a}(D'')=\textbf{a}(V)}} g_{A,D',U}\cdot g_{B,D'',V}\cdot t_{U,V,E},$$
for all $E\in \Theta_{\vartriangle}(n,r).$ By proposition 3.2 and noting that $co(C)=co(B),$ $co(E)=co(D),\ldots,$ this is reduced to an argument for affine Hecke algebra as shown in [L3, 2.4 (b)].

We next show that $\Phi$ preserves the identity element of two algebras. The identity element of $\mathcal{S}_{q}^{\vartriangle}(n,r)$ is $\sum_{\lambda\in \Lambda_{\vartriangle}(n,r)}\theta_{\lambda},$ where $\Lambda_{\vartriangle}(n,r)$ is identified with the subset of diagonal matrices in $\Theta_{\vartriangle}(n,r).$ For $\lambda\in \Lambda_{\vartriangle}(n,r),$ $D\in \textbf{D}_{\vartriangle}(n,r)_{\lambda},$ by Lemma 4.2 we have $$g_{\lambda,D,B}=\begin{cases}1& \hbox {if } D=B; \\0& \hbox {otherwise}.\end{cases}$$
This implies that $$\Phi(\theta_{\lambda})=\sum\limits_{D\in \textbf{D}_{\vartriangle}(n,r)_{\lambda}}t_{D}.$$

Finally, we show that the extended map $\widehat{\Phi}: \mathcal{S}_{q}^{\vartriangle}(n,r)\otimes_{\mathcal{A}}\mathbf{A}\rightarrow \mathcal{J}_{\vartriangle}(n,r)\otimes_{\mathbb{Z}} \mathbf{A}$ is an isomorphism. Write$$\widehat{\Phi}(\theta_{A})=\sum\limits_{B\in \Theta_{\vartriangle}(n,r)}\xi_{A,B} t^{-\textbf{a}(B)+l(w_{0,co(B)})}.$$
By the definition of $\textbf{a}$-function in 2.2 and Proposition 3.2, we have $\xi_{A,B}-\delta_{A,B}\in t\mathbf{A}.$ Therefore, the determinant of $(\xi_{A,B})$ is unequal to 0 in $\mathbf{A}.$ Consequently, $\widehat{\Phi}$ is an isomorphism.
\end{proof}

We call the algebra $\mathcal{J}_{\vartriangle}(n,r)$ the asymptotic form of $\mathcal{S}_{q}^{\vartriangle}(n,r).$\vskip2mm

4.2 Recall from Lemma 2.3 that $\mathcal{J}_{\lambda\mu}$ is the additive subgroup of $\mathcal{J}_{\mathbb{Z}}$ generated by the elements $t_{w}, w\in \mathcal{D}_{\lambda\mu}^{\vartriangle,+}.$ Let $\omega=(\ldots,1^{r},\ldots)$ and define $$T^{(r)}=\bigoplus_{\lambda\in \Lambda_{\vartriangle}(n,r)}\mathcal{J}_{\lambda\omega}.$$
Note that if $n<r,$ $\omega\notin \Lambda_{\vartriangle}(n,r)$ and we then identify $\Lambda_{\vartriangle}(n,r)$ as the subset $\{\lambda\in \Lambda_{\vartriangle}(r,r)|~\lambda_{n+1}=\cdots=\lambda_{r}=0\}$ of $\Lambda_{\vartriangle}(r,r);$ if $n\geq r,$ then $\omega$ is identified with the element $(\ldots,1^{r},0^{n-r},\ldots)$ of $\Lambda_{\vartriangle}(n,r).$ Clearly, as a $\mathbb{Z}$-module, $T^{(r)}$ is the $r$-fold tensor space $V^{\otimes r}$ over a free $\mathbb{Z}$-module $V$ of rank $n.$ Therefore, $T^{(r+r')}\cong T^{(r)}\otimes T^{(r')}$ as $\mathbb{Z}$-modules. By Lemma 2.3 (b), $T^{(r)}$ is a right $\mathcal{J}_{\mathbb{Z}}$-module which can be viewed as the asymptotic form of the $q$-tensor space $\oplus_{\lambda\in \Lambda_{\vartriangle}(n,r)} x_{\lambda}\mathcal{H}_{\vartriangle}(r).$

By analogy with the $q$-Schur algebra case, we can get the following double centralizer property.
\begin{theorem}
Maintain the previous notation. We have

$(a)$ $\mathcal{J}_{\vartriangle}(n,r)\cong \mathrm{End}_{\mathcal{J}_{\mathbb{Z}}}(T^{(r)});$

$(b)$ $\mathcal{J}_{\mathbb{Z}}\cong \mathrm{End}_{\mathcal{J}_{\vartriangle}(n,r)}(T^{(r)}),$ if $n\geq r.$
\end{theorem}
\begin{proof}
Since we have the following isomorphism:$$\mathrm{End}_{\mathcal{J}_{\mathbb{Z}}}(T^{(r)})\cong \bigoplus_{\lambda, \mu\in \Lambda_{\vartriangle}(n,r)} \mathrm{Hom}_{\mathcal{J}_{\mathbb{Z}}}(\mathcal{J}_{\mu\omega}, \mathcal{J}_{\lambda\omega}),$$
by Lemma 2.3 (c), this is a free $\mathbb{Z}$-module with basis $\{\widetilde{t}_{A}|~A\in \Theta_{\vartriangle}(n,r)\},$ where if $A=(\lambda,w_{A},\mu),$ then the action of $\widetilde{t}_{A}$ on $\oplus_{\nu\in \Lambda_{\vartriangle}(n,r), \nu \neq \mu}\mathcal{J}_{\nu\omega}$ is 0, and on $\mathcal{J}_{\mu\omega}$ is the same as $\varphi_{w_{A}^{+}}$ defined in the proof of Lemma 2.3.

Let $\tilde{}$: $\mathcal{J}_{\vartriangle}(n,r)\rightarrow \mathrm{End}_{\mathcal{J}_{\mathbb{Z}}}(T^{(r)})$ be the $\mathbb{Z}$-linear map such that $t_{A}\mapsto \widetilde{t}_{A}.$ We want to prove that $\tilde{}$ is an algebra homomorphism which preserves the identity elements. This can be easily shown by using the following fact: $$\widetilde{t}_{A}\widetilde{t}_{B}=\sum_{y\in \mathcal{D}_{\lambda\mu}^{\vartriangle}}\gamma_{\sigma(A), \sigma(B), y^{+}}\widetilde{t}_{(\lambda,y,\mu)},~~~\mathrm{where}~\lambda=ro(A),\mu=co(B).$$

(b) Note that $\omega\in \Lambda_{\vartriangle}(n,r)$ when $n\geq r.$ Put $e=\sum_{D\in \textbf{D}_{\vartriangle}(n,r)_{\omega}}t_{D}.$ Then $e$ is an idempotent of $\mathcal{J}_{\vartriangle}(n,r)$ and one can easily check that $e\mathcal{J}_{\vartriangle}(n,r)e\cong \mathcal{J}_{\mathbb{Z}}.$ Via this isomorphism, $\mathcal{J}_{\vartriangle}(n,r)e$ becomes a right $\mathcal{J}_{\mathbb{Z}}$-module and is clearly isomorphic to $T^{(r)}.$ Therefore, our assertion is proved.\end{proof}

4.3 The notion of based rings was introduced by Lusztig in [L6]. By definition, an associative ring $R$ with 1 which is a free abelian group with a fixed $\mathbb{Z}$-basis $B$ is called a based ring if\vskip1.5mm

(a) If $b,b'\in B,$ then $bb'=\sum_{b''\in B}n_{b''}b'',$ $n_{b''}\in \mathbb{Z}_{\geq 0}.$

(b) $1=\sum_{b\in B_{0}}b$ for some subset $B_{0}\subset B.$

(c) There exists an anti-automorphism $l$ of $R$ of order 2 satisfying $l(B)=B,$ and such that $\tau(bb')=1$ if $b'=l(b),$ and 0 otherwise. Here $\tau$: $R\rightarrow \mathbb{Z}$ is a group homomorphism defined by $\tau(b)=1$ or 0 if $b\in B_{0}$ or not.\vskip1.5mm

For $A\in \Theta_{\vartriangle}(n,r),$ let $A^{t}$ denote the transpose of $A.$ Note that if $A=(\lambda,w_{A},\mu),$ then $A^{t}=(\mu,w_{A}^{-1},\lambda).$ Thus we have by 2.2 (b), (c) and (d) (see also Proposition 5.1)

(d) $\gamma_{A,B,D}\neq 0, D\in \textbf{D}_{\vartriangle}(n,r)\Rightarrow B=A^{t}, \gamma_{A,A^{t},D}=1;$ for any $A\in \Theta_{\vartriangle}(n,r),$ there exists a unique $D\in \textbf{D}_{\vartriangle}(n,r)$ such that $\gamma_{A,A^{t},D}\neq 0;$ and $\gamma_{A,B,C}=\gamma_{B,C^{t},A^{t}}=\gamma_{C^{t},A,B^{t}}$ for all $A,B,C\in \Theta_{\vartriangle}(n,r).$

Now one can easily check that the following holds.

(e) $\mathcal{J}_{\vartriangle}(n,r)$ with the basis $B=\{t_{A}|~A\in \Theta_{\vartriangle}(n,r)\}$ is a based ring.

Here the set $B_{0}$ is $\{t_{D}|~D\in \textbf{D}_{\vartriangle}(n,r)\},$ and the involution $l$ in (c) is the map sending $t_{A}$ to $t_{A^{t}}.$\vskip2mm

4.4 As a based ring, $\mathcal{J}_{\vartriangle}(n,r)$ has a decomposition into one-sided ideals and into two-sided ideals:
$$\leqno{\mathrm{(a)}} \hspace*{1.3cm}  \mathcal{J}_{\vartriangle}(n,r)=\bigoplus_{D\in \textbf{D}_{\vartriangle}(n,r)}\mathcal{J}_{\vartriangle}(n,r)t_{D}=\bigoplus_{D\in \textbf{D}_{\vartriangle}(n,r)}t_{D}\mathcal{J}_{\vartriangle}(n,r),$$
$$\leqno{\mathrm{(b)}}  \hspace*{4.5cm} \mathcal{J}_{\vartriangle}(n,r)=\bigoplus_{i}\mathcal{J}_{\vartriangle}(n,r)_{i},$$
where $\mathcal{J}_{\vartriangle}(n,r)_{i}$ is spanned by the elements in the equivalence classes $B_{i}$ corresponding to the equivalence relation $\sim$: $b\sim b'\Leftrightarrow bcb'\neq 0$ for some $c\in B$ (see [L6, 1.1 (h)]).

Now we will recall the definition of cells for $\mathcal{S}_{q}^{\vartriangle}(n,r)$ with respect to the canonical basis $\{\theta_{A}|~A\in \Theta_{\vartriangle}(n,r)\}$ (see [Mc]). We define preorder $\preceq_{L}$ on $\Theta_{\vartriangle}(n,r)$ as follows: $A\preceq_{L} B$ if $\theta_{A}$ appears with a nonzero coefficient in the product $\theta_{C}\theta_{B}$ for some $C\in \Theta_{\vartriangle}(n,r),$ and $A\preceq_{R} B$ if $A^{t}\preceq_{L} B^{t}.$ We denote by $A\preceq_{LR} B$ if $\theta_{A}$ appears with a nonzero coefficient in the product $\theta_{C}\theta_{B}\theta_{D}$ for some $C, D\in \Theta_{\vartriangle}(n,r).$ Let $\sim_{L}, \sim_{R}$ and $\sim_{LR}$ be the associated equivalence relations on $\Theta_{\vartriangle}(n,r)$ and call the corresponding equivalence classes the left, right and two-sided cells respectively. We have by 2.4 (a) and (b)

(c) $A\sim_{L} B\Leftrightarrow t_{A}t_{B^{t}}\neq 0;$

(d) $A\sim_{LR} B\Leftrightarrow t_{A}t_{C}t_{B}\neq 0$ for some $C\in \Theta_{\vartriangle}(n,r);$

(e) every left (resp. right) cell $\Gamma$ (resp. $\Gamma'$) of $\Theta_{\vartriangle}(n,r)$ contains a unique element $D$ (resp. $D'$) of $\textbf{D}_{\vartriangle}(n,r),$ and the set $\{t_{A}|~A\in \Gamma\}$ (resp. $\{t_{A'}|~A'\in \Gamma'\}$) forms a basis for the left ideal $\mathcal{J}_{\vartriangle}(n,r)t_{D}$ (resp. $t_{D'}\mathcal{J}_{\vartriangle}(n,r)$) in (a). This follows from 4.3 (d) which implies that $t_{A}t_{D}=t_{A}$ (resp. $t_{D'}t_{A'}=t_{A'}$) for all $A\in \Gamma$ (resp. $A'\in \Gamma'$).

Therefore, the decompositions of one-sided ideals in (a) agrees with the decompositions of left or right cells for $\Theta_{\vartriangle}(n,r),$ while that of two-sided ideals in (b) agrees with the decomposition of two-sided cells.\vskip2mm

4.5 Let $\mathcal{P}_{r}$ be the set of partitions of $r$ and let $\mathcal{P}^{n}_{r}$ be the set of partitions of $r$ with at most $n$ parts. It has been shown by Lusztig, based on the work of Shi in [Shi1], that there is a bijection between the set of two-sided cells of $\mathrm{W}$ and $\mathcal{P}_{r}$ (see [L1]). In fact, this bijection is described by a map $\sigma$ from $\mathrm{W}$ to $\mathcal{P}_{r}$ and the fibers of $\sigma$ are precisely the two-sided cells of $\mathrm{W}$. Similarly, in [Mc], McGerty has defined a map $\rho$ from $\Theta_{\vartriangle}(n,r)$ to $\mathcal{P}_{r}^{n}$ and the fibers of $\rho$ are precisely the two-sided cells of $\Theta_{\vartriangle}(n,r)$ (see [Mc, Proposition 4.4]). Thus we have a bijection between the set of two-sided cells of $\Theta_{\vartriangle}(n,r)$ and $\mathcal{P}_{r}^{n}$. Given a partition $\lambda\in \mathcal{P}_{r}^{n}$, we will denote the two-sided cell $\rho^{-1}(\lambda)$ by $\mathbf{c}_{\lambda}$ in what follows. Thus, from 4.4 (b), we have $\mathcal{J}_{\vartriangle}(n,r)=\bigoplus_{\lambda\in \mathcal{P}_{r}^{n}}\mathcal{J}_{\vartriangle}(n,r)_{\mathbf{c}_{\lambda}}.$

The following proposition allows to count the number of left cells in a two-sided cell of $\mathcal{S}_{q}^{\vartriangle}(n,r).$
\begin{proposition}{\rm (see [Mc, Proposition~4.10])}

Let $\mathbf{c}_{\lambda}$ be a two-sided cell of $\mathcal{S}_{q}^{\vartriangle}(n,r)$ associated with a partition $\lambda\in \mathcal{P}^{n}_{r},$ and let $\lambda(i)=\lambda_{i}-\lambda_{i+1}.$ Then the number of left cells in $\mathbf{c}_{\lambda}$ is $N_{\lambda}=\prod_{i=1}^{n-1} \binom{n}{i}^{\lambda(i)}.$
\end{proposition}

For each $\lambda\in \mathcal{P}^{n}_{r}$ and $i\in \{1,2,\ldots,n\},$ let $\lambda(i)=\lambda_{i}-\lambda_{i+1}$ ($\lambda_{n+1}=0$) as above, and let $G_{\lambda}=\prod_{i=1}^{n}GL_{\lambda(i)}(\mathbb{C}).$ Let $\mathrm{Irr}~ G_{\lambda}$ be the set of irreducible representations of $G_{\lambda}$ and let $B_{\lambda}=R(G_{\lambda})$ be the representation ring of the algebraic group $G_{\lambda}$ with the set $\mathrm{Irr}~ G_{\lambda}$ as a $\mathbb{Z}$-basis.

Let $\textbf{D}_{\mathbf{c}_{\lambda}}=\textbf{D}_{\vartriangle}(n,r)\cap \mathbf{c}_{\lambda},$ and let $T_{\lambda}$ be the set of triples $(E_{1}, E_2, s),$ where $E_1, E_2\in \textbf{D}_{\mathbf{c}_{\lambda}}$ and $s\in \mathrm{Irr}~ G_{\lambda}.$ Let $\mathcal{T}_{\lambda}$ be a free abelian group on the set $T_{\lambda},$ and the ring structure is given as follows: $$(E_1,E_2,s)(E_1',E_2',s')=\delta_{E_2,E_1'}\sum\limits_{s''\in \mathrm{Irr}~ G_{\lambda}}c_{s,s'}^{s''}(E_1,E_2',s'')$$where $c_{s,s'}^{s''}$ is the multiplicity of $s''$ in the tensor product $s\otimes s'.$ Thus, $\mathcal{T}_{\lambda}$ is a matrix ring of rank $N_{\lambda}$ over the representation ring $B_{\lambda}.$

From 4.4 (d), we can get the following proposition by using Xi's results on affine Hecke algebras of type $A$ (see [Xi2]).

\begin{proposition}{\rm (see [Mc, Proposition~4.13])}

$(1)$ There is a ring isomorphism $\mathcal{J}_{\vartriangle}(n,r)_{\mathbf{c}_{\lambda}}\rightarrow \mathcal{T}_{\lambda}$ which restricts to a bijection between the canonical basis of $\mathcal{J}_{\vartriangle}(n,r)_{\mathbf{c}_{\lambda}}$ and $T_{\lambda}.$

$(2)$ For any $E\in \textbf{D}_{\mathbf{c}_{\lambda}},$ the subset of $\mathbf{c}_{\lambda}$ corresponding to $\{(E_1,E_2,s)\in T_{\lambda}|~E_2=E\}$ under the bijection is a left cell.

$(3)$ For any $E\in \textbf{D}_{\mathbf{c}_{\lambda}},$ the subset of $\mathbf{c}_{\lambda}$ corresponding to $\{(E_1,E_2,s)\in T_{\lambda}|~E_1=E\}$ under the bijection is a right cell.
\end{proposition}

\section{Lusztig's conjectures for affine $q$-Schur algebras}

In this section, we prove that the affine $q$-Schur algebra $\mathcal{S}_{q}^{\vartriangle}(n,r)$ satisfies properties very similar to $\textbf{P1}, \textbf{P2},\ldots, \textbf{P15}$ for the affine Hecke algebra $\mathcal{H}_{\vartriangle}(r).$

\begin{lemma}{\rm (see also [Cu2, Proposition~3.1 (6)])}

If $A\preceq_{L} B$ (resp. $A\preceq_{R} B$ or $A\preceq_{LR} B$), then we have $w_{A}^{+}\preceq_{L} w_{B}^{+}$ (resp. $w_{A}^{+}\preceq_{R} w_{B}^{+}$ or $w_{A}^{+}\preceq_{LR} w_{B}^{+}$).
\end{lemma}
\begin{proof}
Since $A\preceq_{L} B,$ there is $C\in \Theta_{\vartriangle}(n,r)$ such that $g_{C,B,A}\neq 0,$ but we have $g_{C,B,A}=h_{co(C)}^{-1}\cdot h_{\sigma(C),\sigma(B),\sigma(A)},$ thus $h_{\sigma(C),\sigma(B),\sigma(A)}\neq 0$ and we get $w_{A}^{+}\preceq_{L} w_{B}^{+}.$
\end{proof}

\begin{lemma}If $A\preceq_{L} B$, then $co(A)=co(B);$ if $A\preceq_{R} B,$ then $ro(A)=ro(B).$
\end{lemma}
\begin{proof}

Since $A\preceq_{L} B,$ there is $C\in \Theta_{\vartriangle}(n,r)$ such that $g_{C,B,A}\neq 0.$ From Proposition 3.2, it follows that $(ro(A), co(A))=(ro(C), co(B)),$ and we get the result.
\end{proof}

\begin{lemma}
Let $\lambda,\mu,\nu\in \Lambda_{\vartriangle}(n,r),$ $x\in \mathcal{D}_{\lambda\mu}^{\vartriangle,+}$ and $y\in \mathcal{D}_{\mu\nu}^{\vartriangle,+},$ if $h_{x,y,z}\neq 0$ for some $z\in W,$ then $z\in \mathcal{D}_{\lambda\nu}^{\vartriangle,+}.$
\end{lemma}
\begin{proof}
For $\lambda\in \Lambda_{\vartriangle}(n,r),$ we set $S_{\lambda}$: $=W_{\lambda}\cap S,$ the set of Coxeter generators of the parabolic subgroup $W_{\lambda}.$ Let $x\in \mathcal{D}_{\lambda\mu}^{\vartriangle,+},$ $y\in \mathcal{D}_{\mu\nu}^{\vartriangle,+}$ and $h_{x,y,z}\neq 0.$ On the one hand, this means that $sx< x$ for all $s\in S_{\lambda}$ and $ys< y$ for all $s\in S_{\nu}.$ On the other hand, we get $z\preceq_{L} y$ and $z\preceq_{R} x,$ and thus $zs< z$ for all $s\in S$ with $ys< y$ and $sz< z$ for all $s\in S$ with $sx< x$ by 2.3 (a). Thus, we have in particular that $zs<z$ for all $s\in S_{\nu}$ and $sz< z$ for all $s\in S_{\lambda}.$ Hence $z$ is the longest element in its $W_{\lambda}$-$W_{\nu}$-double coset in $W.$
\end{proof}
\begin{lemma}
We have $A\preceq_{R} B$ if and only if there is $C\in \Theta_{\vartriangle}(n,r)$ with $g_{B,C,A}\neq 0.$
\end{lemma}
\begin{proof}
$A\preceq_{R} B$ is equivalent to $A^{t}\preceq_{L} B^{t}.$ This in turn means that there is $C\in \Theta_{\vartriangle}(n,r)$ such that $g_{C, B^{t}, A^{t}}\neq 0.$ From $h_{x,y,z}=h_{y^{-1},x^{-1},z^{-1}}$ for all $x,y,z\in W$ together with $\sigma(A^{t})=\sigma(A)^{-1},$ we get that $g_{B,C,A}=0$ if and only if $g_{C^{t}, B^{t}, A^{t}}=0$ from Proposition 3.2, which directly implies the result.
\end{proof}

We now recall the following result of Mcgerty (see [Mc, Proposition 3.8 (1), (2) and (3)]).
\begin{lemma}
$(1)$   $A\sim_{LR} B$ if and only if $w_{A}^{+}\sim_{LR} w_{B}^{+}.$

$(2)$   $A\sim_{L} B$ if and only if $w_{A}^{+}\sim_{L} w_{B}^{+}$ and $co(A)=co(B).$

$(3)$   $A\sim_{R} B$ if and only if $w_{A}^{+}\sim_{R} w_{B}^{+}$ and $ro(A)=ro(B).$
\end{lemma}

For any $y, w\in W,$ we define $v^{-l(w)}P_{y,w}=p_{y, w},$ where $p_{y,w}$ is just as in [L8, $\S5.3$]. Moreover, for $z\in W,$ if we define $\Delta(z)=-$deg$p_{1,z}$ following [L8, $\S14.1$], then we have $-\Delta(z)=-l(z)+2\delta(z),$ i.e. $\Delta(z)=l(z)-2\delta(z).$

Lusztig has proved that Conjectures 14.2 in [L8] hold in a number of cases, including the case of a finite Weyl group, an affine Weyl group, or a dihedral group. In [BN, Proposition 3.7], they have proved that the analogous statements for the $q$-Schur algebra $\mathcal{S}_{q}(n,r)$ hold. Now we will prove that the analogous statements for the affine $q$-Schur algebra also hold.
\begin{proposition}
The following properties hold for the affine $q$-Schur algebra $\mathcal{S}_{q}^{\vartriangle}(n,r):$

\textbf{Q1} ~~For any $A\in \Theta_{\vartriangle}(n,r)$ we have $\textbf{a}(A)\leq \Delta(\sigma(A)).$

\textbf{Q2} ~~Let $A,B\in \Theta_{\vartriangle}(n,r),$ if $\gamma_{A,B,D}\neq 0$ for some $D\in \textbf{D}_{\vartriangle}(n,r),$ then we have $B=A^{t}.$

\textbf{Q3} ~~If $A\in \Theta_{\vartriangle}(n,r),$ there is a unique element $D\in \textbf{D}_{\vartriangle}(n,r)$ with $\gamma_{A^{t},A,D}\neq 0.$

\textbf{Q4} ~~If $A\preceq_{LR} B,$ then $\textbf{a}(A)\geq \textbf{a}(B).$ Hence, if $A\sim_{LR} A',$ then $\textbf{a}(A)= \textbf{a}(A').$

\textbf{Q5} ~~If $D\in \textbf{D}_{\vartriangle}(n,r)$ and $A\in \Theta_{\vartriangle}(n,r)$ are such that $\gamma_{A^{t},A,D}\neq 0,$ then $\gamma_{A^{t},A,D}=1.$

\textbf{Q6} ~~For $D\in \textbf{D}_{\vartriangle}(n,r),$ we have $D=D^{t}.$

\textbf{Q7} ~~For any $A,B,C\in \Theta_{\vartriangle}(n,r),$ we have $\gamma_{A,B,C}=\gamma_{B,C^{t},A^{t}}=\gamma_{C^{t},A,B^{t}}.$

\textbf{Q8} ~~Let $A,B,C\in \Theta_{\vartriangle}(n,r)$ be such that $\gamma_{A,B,C}\neq 0,$ then $A\sim_{L} B^{t}, B\sim_{L} C$ and $A\sim_{R} C.$

\textbf{Q9} ~~If $A\preceq_{L} B$ and $\textbf{a}(A)= \textbf{a}(B),$ then $A\sim_{L} B.$

\textbf{Q10} ~~If $A\preceq_{R} B$ and $\textbf{a}(A)= \textbf{a}(B),$ then $A\sim_{R} B.$

\textbf{Q11} ~~If $A\preceq_{LR} B$ and $\textbf{a}(A)= \textbf{a}(B),$ then $A\sim_{LR} B.$

\textbf{Q13} ~~Every left cell contains a unique element $D\in \textbf{D}_{\vartriangle}(n,r).$ We have $\gamma_{A^{t},A,D}\neq 0$ for all $A\sim_{L} D.$

\textbf{Q14} ~~For every $A\in \Theta_{\vartriangle}(n,r),$ we have $A\sim_{LR} A^{t}.$

\textbf{Q15} ~~Let $v'$ be a second indeterminate and let $g_{A,B,C}'\in \mathbb{Z}[v', v'^{-1}]$ be obtained from $g_{A,B,C}$ by the substitution $v\mapsto v'.$ If $A,A',B,C\in \Theta_{\vartriangle}(n,r)$ satisfy $\textbf{a}(C)= \textbf{a}(B),$ then $$\sum_{B'}g_{C,A',B'}'g_{A,B',B}=\sum_{B'}g_{A,C,B'}g_{B',A',B}'.$$
\end{proposition}

\begin{proof}
$\textbf{Q1}$ is a direct consequence of 2.2 (a).

$\textbf{Q2}$ Suppose that $\gamma_{A,B,D}\neq 0$ for some $A, B\in \Theta_{\vartriangle}(n,r)$ and $D\in \textbf{D}_{\vartriangle}(n,r).$ Since $\gamma_{A,B,D}\neq 0,$ it follows that $g_{A,B,D}\neq 0.$ Thus we have $co(A)=ro(B), ro(A)=ro(D)$ and $co(B)=co(D);$ but $co(D)=ro(D)$ implies that $ro(A)=co(B).$ We now write $A=(\lambda,w_{A},\mu),$ $B=(\mu,w_{B},\lambda).$ We have $\gamma_{A,B,D}=\gamma_{\sigma(A),\sigma(B),\sigma(D)},$ from $\sigma(D)\in \textbf{D}$, we deduce from 2.2 (b) that $\sigma(A)=\sigma(B)^{-1},$ thus, we get $w_{A}=w_{B}^{-1},$ so $B=A^{t}.$

$\textbf{Q3}$ For $A=(\lambda,w_{A},\mu)\in \Theta_{\vartriangle}(n,r),$ by 2.2 (c), there is a unique $d\in \textbf{D}$ such that $\gamma_{\sigma(A)^{-1},\sigma(A),d}\neq 0.$ Since $\sigma(A)^{-1}=\sigma(A^{t}),$ we deduce that $h_{\sigma(A^{t}), \sigma(A),d}\neq 0,$ but $\sigma(A^{t})\in \mathcal{D}_{\mu\lambda}^{\vartriangle,+}$ and $\sigma(A)\in \mathcal{D}_{\lambda\mu}^{\vartriangle,+},$ then Lemma 5.3 implies that $d\in \mathcal{D}_{\mu\mu}^{\vartriangle,+}.$ We denote by $\tilde{d}$ the representative of minimal length of the coset $W_{\mu}dW_{\mu},$ and we set $D$: $=(\mu,\tilde{d},\mu),$ then $D\in \textbf{D}_{\vartriangle}(n,r)$ and $\sigma(D)=d.$ It follows that $\gamma_{A^{t}, A, D}=\gamma_{\sigma(A^{t}), \sigma(A),\sigma(D)}\neq 0.$

$\textbf{Q4}$ follows from 2.3 (b) and Lemma 5.1.

$\textbf{Q5}$ follows from 2.2 (b).

$\textbf{Q6}$ If $D=(\lambda,w_{D},\lambda)\in \textbf{D}_{\vartriangle}(n,r),$ then we have $\sigma(D)\in \textbf{D},$ thus, $\sigma(D)^{-1}=\sigma(D)$ by 2.2 (b). Therefore, $w_{D}^{-1}=w_{D}$ and $D=D^{t}.$

$\textbf{Q7}$ follows from 2.2 (d).

$\textbf{Q8}$ Suppose that $\gamma_{A,B,C}\neq 0$ for some $A,B,C\in \Theta_{\vartriangle}(n,r),$ then from $h_{co(A)}\cdot g_{A,B,C}=h_{\sigma(A),\sigma(B),\sigma(C)},$ we get $g_{A,B,C}\neq 0,$ so $C\preceq_{L} B,$ $C\preceq_{R} A.$ Similarly, from $\gamma_{B,C^{t},A^{t}}\neq 0,$ we get $A^{t}\preceq_{R} B,$ $A^{t}\preceq_{L} C^{t};$ from $\gamma_{C^{t},A,B^{t}}\neq 0,$ we get $B^{t}\preceq_{R} C^{t},$ $B^{t}\preceq_{L} A.$ So we get $A\sim_{L} B^{t}, B\sim_{L} C$ and $A\sim_{R} C.$

$\textbf{Q9}$ Let $A,B\in \Theta_{\vartriangle}(n,r)$ with $A\preceq_{L} B$ and $\textbf{a}(A)= \textbf{a}(B).$ By Lemma 5.1, we have $\sigma(A)\preceq_{L} \sigma(B),$ hence 2.3 (d) implies that $\sigma(A)\sim_{L} \sigma(B).$ Moreover, by Lemma 5.2, we have $co(A)=co(B).$ So $A\sim_{L} B$ by Lemma 5.5 (2).

$\textbf{Q10}$ can be proved similarly or follows from $\textbf{Q9}$ using $\textbf{a}(A)= \textbf{a}(A^{t})$ for all $A\in \Theta_{\vartriangle}(n,r).$

$\textbf{Q11}$ can also be proved similarly.

$\textbf{Q13}$ Let $A\in \Theta_{\vartriangle}(n,r).$ By $\textbf{Q3},$ there is a unique element $D\in \textbf{D}_{\vartriangle}(n,r)$ with $\gamma_{A^{t},A,D}\neq 0$ and for this $D$ we have $A\sim_{L} D$ by $\textbf{Q8}.$ If we have $D\sim_{L} D'$ for $D, D'\in \textbf{D}_{\vartriangle}(n,r),$ then we conclude that $ro(D)=co(D)=co(D')=ro(D')$ using Lemma 5.2 and $\sigma(D)\sim_{L} \sigma(D')$ using Lemma 5.1. So we have $\sigma(D)=\sigma(D')$ by 2.3 (f) and thus we get $D=D'.$

$\textbf{Q14}$  Let $A\in \Theta_{\vartriangle}(n,r)$ and $D\in \textbf{D}_{\vartriangle}(n,r)$ be the unique element such that $A\sim_{L} D$ by $\textbf{Q13},$ then $A^{t}\sim_{R} D^{t}=D,$ so we get $A\sim_{LR} A^{t}.$

$\textbf{Q15}$ If $g_{C,A',B'}'\neq 0,$ then we have $g_{C,A',B'}'=h_{ro(A^{'})}'^{-1}\cdot h'_{\sigma(C),\sigma(A'),\sigma(B')};$ if $g_{A,B',B}\neq 0,$ then $g_{A,B',B}=h_{co(A)}^{-1}\cdot h_{\sigma(A),\sigma(B'),\sigma(B)};$ if $g_{A,C,B'}\neq 0,$ then $g_{A,C,B'}=h_{co(A)}^{-1}\cdot h_{\sigma(A),\sigma(C),\sigma(B')};$ if $g_{B',A',B}'\neq 0,$ then $g_{B',A',B}'=h_{ro(A^{'})}'^{-1}\cdot h'_{\sigma(B'),\sigma(A'),\sigma(B)},$ here $h_{\mu}'$ is obtained from $h_{\mu}$ by the substitution $v\mapsto v'.$ It follows from [L3, 2.4 (c)] that
\begin{align*}
\sum_{B'}g_{C,A',B'}'g_{A,B',B}&=h_{ro(A^{'})}'^{-1}h_{co(A)}^{-1}\sum_{B'}h'_{\sigma(C),\sigma(A'),\sigma(B')}h_{\sigma(A),\sigma(B'),\sigma(B)}\end{align*}
\begin{align*}
&~~~~~~~~~~~~=h_{ro(A^{'})}'^{-1}h_{co(A)}^{-1}\sum_{B'}h_{\sigma(A),\sigma(C),\sigma(B')}h'_{\sigma(B'),\sigma(A'),\sigma(B)}\\
&~~~~~~~~~~~~=\sum_{B'}g_{A,C,B'}g_{B',A',B}'.
\end{align*}
So we have proved the identity. \end{proof}

\section{Some examples of affine $q$-Schur algebras}

6.1 Let $(W_{0}, S_{0})$ denote the finite Weyl group associated with $W,$ that is, $S_{0}=\{s_{i}\}_{1\leq i\leq r-1}$ and $W_{0}=\mathfrak{S}_{r}.$ For each element $w$ in $W,$ we have $\textbf{a}(w)\leq \textbf{a}(w_{0})=l(w_{0})=\nu,$ where $w_{0}$ is the longest element of $W_{0}$ (see [L2]). The set $c_{0}=\{w\in W|~\textbf{a}(w)=l(w_{0})=\nu\}$ is a two-sided cell of $W.$ The two-sided cell contains $|W_{0}|=r!$ left (resp. right) cells, which is the lowest one in the set Cell($W$) of two-sided cells of $W$ concerned with the partial order $\preceq_{LR}$ (see [Shi2, Shi3]). The two-sided cell $c_{0}$ corresponds to the nilpotent $G$-orbit $\{0\}$ under Lusztig's bijection between the set Cell($W$) and the set of nilpotent $G$-orbits in $g=$Lie($G$), where $G$ is the connected reductive group over $\mathbb{C}$ with Weyl group $W_{0}.$ Under the bijection between the set Cell($W$) and the set $\mathcal{P}_{r}$, $c_{0}$ also corresponds to the partition $(r).$

We now consider the two-sided cell $c_{0}^{\vartriangle}$ of $\mathcal{S}_{q}^{\vartriangle}(n,r)$ corresponding to $c_{0},$ i.e. $c_{0}^{\vartriangle}=\{A\in \Theta_{\vartriangle}(n,r)|~\textbf{a}(A)=\textbf{a}(\sigma(A))=\nu\}.$ According to Lemma 5.1, $c_{0}^{\vartriangle}$ is just the lowest one in the set Cell($\Theta_{\vartriangle}(n,r)$) of two-sided cells in $\Theta_{\vartriangle}(n,r)$ and also corresponds to the partition $(r)$ under the bijection between the set Cell($\Theta_{\vartriangle}(n,r)$) and the set $\mathcal{P}^{n}_{r}.$ From Proposition 4.1, we know that the number of left cells in $c_{0}^{\vartriangle}$ is $n^{r}.$
\begin{lemma}
For $A=(\lambda,w_{A},\mu),$ $B=(\mu,1,\nu)\in \Theta_{\vartriangle}(n,r)$ with $W_{\nu}\subset W_{\mu},$ we define $C=(\lambda,w',\nu),$ where $w_{A}^{+}\in W_{\lambda}w'W_{\nu},$ then $\theta_{A}\theta_{B}=\theta_{C}.$
\end{lemma}
\begin{proof}
It follows from the following identity:\vskip2mm $\hspace*{3cm}C_{w_{A}^{+}}C_{w_{0,\mu}}=t^{-l(w_{0,\mu})}\sum\limits_{w\in W_{\mu}}t^{2l(w)}C_{w_{A}^{+}}.$
\end{proof}

6.2 We first consider the affine $q$-Schur algebra $\mathcal{S}_{q}^{\vartriangle}(1,r).$ According to 4.5, it has only one two-sided cell $c_{0}^{\vartriangle},$ which corresponds to the partition $\mu=(r)\in \mathcal{P}^{1}_{r},$ then we have $G_{\mu}=GL_{r}(\mathbb{C})$ and $B_{\mu}=R(G_{\mu})\cong \mathbb{Z}[X_{1},\ldots,X_{r-1},X_{r},X_{r}^{-1}]$ (for example, see [FH, Exercise 23.36 (d). p. 379]). From Proposition 4.2 (1), we get the following isomorphism: $$\mathcal{J}_{\vartriangle}(1,r)\cong \mathbb{Z}[X_{1},\ldots,X_{r-1},X_{r},X_{r}^{-1}].$$ We will denote $\mathcal{J}_{\vartriangle}(1,r)$ by $\mathcal{J}$ in the following.

From Theorem 4.1, we have an $\mathcal{A}$-module homomorphism $\Phi: \mathcal{S}_{q}^{\vartriangle}(1,r)\rightarrow \mathcal{J},$ so each $\mathcal{J}$-module $E$ is endowed with an $\mathcal{S}_{q}^{\vartriangle}(1,r)$-module structure through $\Phi,$ we will denote the $\mathcal{S}_{q}^{\vartriangle}(1,r)$-module by $E_{\Phi}.$

In [Cu2, Theorem 4.1], we have shown that the affine $q$-Schur algebra $\mathcal{S}_{q}^{\vartriangle}(1,r)$ is affine cellular in the sense of Koenig and Xi (see [KX]). Applying [KX, Theorem 4.1], we can get the following corollary.
\begin{corollary}
$(a)$ Over a noetherian domain $k,$ for each simple $\mathcal{J}$-module $E$, the associated $\mathcal{S}_{q}^{\vartriangle}(1,r)$-module $E_{\Phi}$ has a composition factor, we denote it by $M.$

$(b)$ Keep the assumption and notation in $(a),$ the map $E\mapsto M$ defines a bijection between the isomorphism classes of simple $\mathcal{J}$-modules and the isomorphism classes of simple $\mathcal{S}_{q}^{\vartriangle}(1,r)$-modules.
\end{corollary}

When $k=\mathbb{C}$ is the complex number field, from Proposition 4.2, we get that all the simple modules of the $\mathbb{C}$-algebra $\mathbb{C}\otimes_{\mathbb{Z}} \mathcal{J}$ all have dimension $1,$ and the set of isomorphism classes of such modules is in bijection with the semisimple conjugacy classes of $G_{\mu}=GL_{r}(\mathbb{C}).$ Recall that $\mathcal{J}\cong R(G_{\mu}),$ for each semisimple conjugacy class $s$ in $GL_{r}(\mathbb{C}),$ we have a simple representation $\psi_{s}$ of $\mathcal{J}$:$$\psi_{s}: R(G_{\mu})\rightarrow \mathbb{C},~~~m\mapsto tr(s, m).$$
Any simple representation of $\mathcal{J}$ over $\mathbb{C}$ is isomorphic to some $\psi_{s}$ (see [Xi1]). So we have
\begin{corollary}
The isomorphism classes of simple $\mathcal{S}_{q}^{\vartriangle}(1,r)$-modules over $\mathbb{C}$ are parametrized by the semisimple conjugacy  classes in $GL_{r}(\mathbb{C}).$
\end{corollary}

For the affine $q$-Schur algebra $\mathcal{S}_{q}^{\vartriangle}(n,1),$ since it has only one two-sided cell, we have a similar discussion as above.

Applying [KX, Theorem 4.4], we can get the following corollary.
\begin{corollary}
Let $\mathcal{S}$ stands for the affine $q$-Schur algebra $\mathcal{S}_{q}^{\vartriangle}(1,r)$ or $\mathcal{S}_{q}^{\vartriangle}(n,1).$ Let $k$ be a noetherian domain and $\mathcal{S}_{k}=k\otimes_{\mathbb{Z}}\mathcal{S}.$ Then we get

$(a)$ The parameter set of simple $\mathcal{S}_{k}$-modules equals the parameter set of simple modules of the asymptotic algebra, so it is an affine space.

$(b)$ The unbounded derived module category $D(\mathcal{S}_{k}$-$\mathrm{Mod})$ of $\mathcal{S}_{k}$ admits a stratification whose section is the derived category of the affine $k$-algebra $B.$

$(c)$ $\mathcal{S}_{k}$ has finite global dimension provided that $k$ has that.
\end{corollary}

6.3 In this subsection, we consider the affine $q$-Schur algebra $\mathcal{S}_{q,k}^{\vartriangle}(2,2)$ over a field $k,$ where $\mathcal{S}_{q,k}^{\vartriangle}(2,2)=k\otimes_{\mathcal{A}}\mathcal{S}_{q}^{\vartriangle}(2,2)$ and $k$ is regarded as an $\mathcal{A}$-module by specializing $t$ to a square root of $q$ belonging to $k.$ According to 4.5, it has two two-sided cells, the lowest two-sided cell $c_{0}^{\vartriangle}$ corresponds to the partition $(2)\in \mathcal{P}^{2}_{2}$ and we have $c_{0}^{\vartriangle}=\{A\in \Theta_{\vartriangle}(2,2)|~\textbf{a}(A)=1\}.$ Let $\mathcal{S}_{q,c_{0}^{\vartriangle}}^{\vartriangle}$ be the lowest two-sided ideal of $\mathcal{S}_{q,k}^{\vartriangle}(2,2),$ which is the free $k$-module spanned by $\theta_{A},$ $A\in c_{0}^{\vartriangle},$ then we get the following lemma.
\begin{lemma}
when char~$k=0$ and $1+q\neq 0,$ the lowest two-sided ideal $\mathcal{S}_{q,c_{0}^{\vartriangle}}^{\vartriangle}$ is idempotent and has a nonzero idempotent element.
\end{lemma}
\begin{proof}
For any $A=(\lambda,w_{A},\nu)\in c_{0}^{\vartriangle},$ we choose $\mu=(2)\in \Lambda_{\vartriangle}(2,2),$ then $D'=(\mu,1,\nu)\in c_{0}^{\vartriangle}.$ If $\nu=(2)\in \Lambda_{\vartriangle}(2,2),$ then from Lemma 6.1, we have $\theta_{A}\theta_{D'}=\theta_{A}.$ If $\lambda=(2)\in \Lambda_{\vartriangle}(2,2),$ then from Lemma 4.2, we have $\theta_{D''}\theta_{A}=\theta_{A},$ where $D''=(\lambda,1,\mu)\in c_{0}^{\vartriangle}.$ If $\lambda=\nu=\omega=(11)\in \Lambda_{\vartriangle}(2,2),$ then $A=(\omega,w_{A},\omega)$ can be identified with an element $w_{A}$ in $\mathcal{H}_{\vartriangle}(2)$ and $\{A\}$ is identified with $C_{w_{A}}.$ when char~$k=0$ and $1+q\neq 0,$ applying [KX, Theorem 5.7] and [Xi3, Theorem 3.2 (a)], we see that $\{A\}$ can also be written as a product of two elements in $\mathcal{S}_{q,c_{0}^{\vartriangle}}^{\vartriangle}.$ For $\mu=(2)\in \Lambda_{\vartriangle}(2,2),$ then $D=(\mu,1,\mu)\in c_{0}^{\vartriangle}$ satisfies $\theta_{D}\theta_{D}=\theta_{D}.$
\end{proof}

In [Cu2, Theorem 4.1], we have shown that the affine $q$-Schur algebra $\mathcal{S}_{q}^{\vartriangle}(2,2)$ is affine cellular in the sense of Koenig and Xi. Applying [KX, Theorem 4.3 and 4.4], from Lemma 6.2 we get the following corollary.
\begin{corollary}
Assume that char~$k=0$ and $1+q\neq 0.$ Then we get

$(a)$ The parameter set of simple $\mathcal{S}_{q,k}^{\vartriangle}(2,2)$-modules equals the parameter set of simple modules of the asymptotic algebra, so it is a finite union of affine spaces.

$(b)$ The unbounded derived module category $D(\mathcal{S}_{q,k}^{\vartriangle}(2,2)$-$\mathrm{Mod})$ of $\mathcal{S}_{q,k}^{\vartriangle}(2,2)$ admits a stratification whose strata are the derived categories of the various affine $k$-algebras $B_{\lambda}.$

$(c)$ $\mathcal{S}_{q,c_{0}^{\vartriangle}}^{\vartriangle}$ is a projective $\mathcal{S}_{q,k}^{\vartriangle}(2,2)$-module and $\mathcal{S}_{q,k}^{\vartriangle}(2,2)$ has finite global dimension.
\end{corollary}

\hspace*{-0.5cm}$\mathbf{Remark ~6.1.}$ The case of $\mathcal{S}_{q,k}^{\vartriangle}(n,1), n> 1$ has also been considered in [Cu1] from the quantum group approach. In [Cu1], we have proved that the affine $q$-Schur algebras $\mathcal{S}_{q,k}^{\vartriangle}(n,r)$ over a noetherian domain $k,$ when $n> r,$ have finite global dimension provided that $k$ has that, and we conjecture that they always have finite global dimension. Furthermore, we conjecture that the finite global dimension is always even, I am grateful to Professor Changchang Xi for pointing it out.

\section{The identification of two bases}

7.1 We first give a geometric interpretation of the affine $q$-Schur algebra $\mathfrak{U}_{r,n,n}$ following [L7] (see also [GV]). Thus, let $V_{\epsilon}$ be a free $\mathbf{k}[\epsilon,\epsilon^{-1}]$-module of rank $r$, where $\mathbf{k}$ is a finite field of $q$ elements, and $\epsilon$ is an indeterminate. A lattice in $V_{\epsilon}$ is, by definition, a $\mathbf{k}[\epsilon]$-submodule $L$ of $V_{\epsilon}$ such that there exists a $\mathbf{k}[\epsilon]$-basis of $L$ which is also a $\mathbf{k}[\epsilon,\epsilon^{-1}]$-basis of $V_{\epsilon}.$

Let $\mathcal{F}^{n}$ be the space of $n$-step periodic lattices, that is, sequences $\mathbf{L}=(L_{i})_{i\in \mathbb{Z}}$ of lattices in our free module $V_{\epsilon}$ such that $L_{i} \subset L_{i+1}$ and $L_{i-n}=\epsilon L_{i}$. The group $\mathrm{G=Aut}(V_{\epsilon})$ acts on $\mathcal{F}^{n}$ in the natural way. Let $\mathfrak{S}_{r, n}$ be the set of nonnegative integer sequences $(a_{i})_{i\in \mathbb{Z}}$ such that $a_{i}=a_{i+n}$ and $\sum_{i=1}^{n}a_{i}=r$, and let $\mathfrak{S}_{r, n,n}$ be the set of $\mathbb{Z}\times \mathbb{Z}$ matrices $A=(a_{i,j})_{i,j\in \mathbb{Z}}$ with nonnegative entries such that $a_{i,j}=a_{i+n,j+n}$ and $\sum_{i\in [1,n],j\in \mathbb{Z}}a_{i,j}=r.$ The orbits of $\mathrm{G}$ on $\mathcal{F}^{n}$ are indexed by $\mathfrak{S}_{r, n}$, where $\mathbf{L}$ is in the orbit $\mathcal{F}_{\mathbf{a}}$ corresponding to $\mathbf{a}$ if $a_{i}=\mathrm{dim}_{\mathbf{k}}(L_{i}/L_{i-1})$. The orbits of $\mathrm{G}$ on $\mathcal{F}^{n}\times \mathcal{F}^{n}$ are indexed by the matrices $\mathfrak{S}_{r, n,n}$, where a pair $(\mathbf{L},\mathbf{L}')$ is in the orbit $\mathcal{O}_{A}$ corresponding to $A$ if $$a_{i,j}=\mathrm{dim}~\bigg(\frac{L_{i}\cap L_{j}'}{(L_{i-1}\cap L_{j}')+(L_{i}\cap L_{j-1}')}\bigg).$$
For $A\in \mathfrak{S}_{r, n,n},$ let $r(A), c(A)\in \mathfrak{S}_{r, n}$ be given by $r(A)_{i}=\sum_{j\in \mathbb{Z}}a_{i,j}$ and $c(A)_{j}=\sum_{i\in \mathbb{Z}}a_{i,j}.$

Similarly, let $\mathcal{B}^{r}$ be the space consisting of complete periodic lattices, that is, sequences of $\mathbf{L}=(L_{i})_{i\in \mathbb{Z}}$ such that $L_{i} \subset L_{i+1}$ and $L_{i-r}=\epsilon L_{i},$ and $\mathrm{dim}_{\mathbf{k}}(L_{i}/L_{i-1})=1$ for all $i\in \mathbb{Z}$. Let $\mathbf{b}_{0}=(\cdots,1,1,\cdots)$. The orbits of $\mathrm{G}$ on $\mathcal{B}^{r}\times \mathcal{B}^{r}$ are indexed by matrices $\mathfrak{S}_{r, r,r}$, where the matrix $A$ must have $r(A)=c(A)=\mathbf{b}_{0}.$

Let $\mathfrak{U}_{r, q}$ be the span of the characteristic functions of the $\mathrm{G}$-orbits on $\mathcal{F}^{n}\times \mathcal{F}^{n}$, then convolution makes $\mathfrak{U}_{r, q}$ an algebra. For any $A, A',A''\in \mathfrak{S}_{r, n,n}$ and any $q,$ a power of prime number, we denote by $g_{A,A',A'';q}$ the number of elements in the (finite) set $$\{L'\in \mathcal{F}_{co(A)}|~(L,L')\in \mathcal{O}_{A}, (L',L'')\in \mathcal{O}_{A'}\},~~\mathrm{where}~(L,L'')~\mathrm{is~fixed~in~}\mathcal{O}_{A''}.$$
It is a well-known property that we can find $g_{A,A',A''}\in \mathcal{A}$ such that $g_{A,A',A'';q}=g_{A,A',A''}|_{v=q^{\frac{1}{2}}}$ for any prime power $q.$

Following [L7], we define $\mathfrak{U}_{r,n,n}$ to be the free $\mathcal{A}$-module with basis $\{e_{A}|~A\in \mathfrak{S}_{r, n,n}\},$ and multiplication defined by $$e_{A}e_{A'}=\begin{cases}\sum\limits_{A''\in \mathfrak{S}_{r, n,n}}g_{A,A',A''}e_{A''}& \hbox {if } co(A)=ro(A'); \\~~~~~~~~~~0& \hbox {otherwise}.\end{cases}$$

\begin{lemma}
The map $\Upsilon: \mathfrak{U}_{r,n,n}\rightarrow \mathcal{S}_{q}^{\vartriangle}(n,r),$ $e_{A}\mapsto \phi_{A}$ is an isomorphism of algebras.
\end{lemma}
Under this isomorphism $\Upsilon$, we will identify $\mathfrak{S}_{r, n}$ with $\Lambda_{\vartriangle}(n,r),$ $\mathfrak{S}_{r, n,n}$ with $\Theta_{\vartriangle}(n,r).$ For $A\in \mathfrak{S}_{r, n,n}$ and $L\in \mathcal{F}_{ro(A)},$ we set$$X_{A}^{L}=\{L'\in \mathcal{F}_{co(A)}|~(L,L')\in \mathcal{O}_{A}\},~~~d_{A}=\sum\limits_{i\geq k, j<l, 1\leq i\leq n}a_{i,j}a_{k,l}.$$
Then we have dim$(X_{A}^{L})=$dim$(\bar{X}_{A}^{L})=d_{A}$ by [L7, Lemma 4.3]. If we set $[A]=v^{-d_{A}}e_{A},$ then $\{[A]|~A\in \mathfrak{S}_{r, n,n}\}$ is also an $\mathcal{A}$-basis of $\mathfrak{U}_{r,n,n}.$\vskip2mm

7.2 We shall, from now on, assume that $\mathbf{k}$ is an algebraically closed field. We can keep the importance case of a finite flag variety in mind if we are unfamiliar with Kac-Moody flag varieties, the standard reference for Kac-Moody flag varieties is [Kum]. For each $\lambda\in \Lambda_{\vartriangle}(n,r),$ we denote by $\mathrm{P}_{\lambda}$ the standard parabolic subgroup of $\mathrm{G}=GL_{r}(\mathbf{k}[\epsilon,\epsilon^{-1}]),$ which is defined as follows: given $\lambda\in \Lambda_{\vartriangle}(n,r),$ we can get a set $I=\{(i_1,i_2,\ldots,i_n)|~0=i_1<i_2<\cdots <i_n\}\subseteq [0,r-1],$ where $i_{j}=\lambda_1+\cdots+\lambda_{j-1},$ if we set $i_{n+1}=r,$ then we also have $\lambda_{j}=i_{j+1}-i_{j}.$ Consider the complement $\hat{I}=[0, r-1]\backslash I,$ then the partial flag variety corresponding to $\hat{I}$ is:\vskip3mm $~\mathrm{G}/\mathrm{P}_{\lambda}\cong \bigg\{(\cdots\supseteq \Lambda_{n}\supseteq\Lambda_{n-1}\supseteq\cdots\supseteq\Lambda_{1}\supseteq\cdots)\bigg|\substack{\Lambda_{j} \mathrm{~a ~lattice,~} \Lambda_{j-n}=\epsilon\Lambda_{j},\\ \mathrm{dim}_{\mathbf{k}}\big(\Lambda_{j}/\Lambda_{j-1}\big)=\lambda_{j}} \bigg\}.$\vskip3mm

Let $e_1,e_2,\ldots,e_{r}$ denote the standard $\mathbf{k}[\epsilon,\epsilon^{-1}]$-basis of $V_{\epsilon},$ and for $c\in \mathbb{Z},$ define $e_{i+cr}$: $=\epsilon^{c}e_{i}.$ Consider the following family of standard $\mathbf{k}[\epsilon]$-lattices: $E_{j}$: $=$ span$_{\mathbf{k}[\epsilon]}\{e_{r-j+1},\ldots,e_{2r-j}\}$ for $j=1,2,\ldots,r.$ The affine Grassmannian $Gr(V_{\epsilon})$ is the span of all $\mathbf{k}[\epsilon]$-lattices of $V_{\epsilon},$ which is a homogeneous space with respect to the obvious action of $\mathrm{G},$ and the stabilizer of the standard lattice $E_{r}$ is $\mathrm{P}_{(r)}$: $=GL_{r}(\mathbf{k}[\epsilon]),$ thus we have $Gr(V_{\epsilon})\cong \mathrm{G}/\mathrm{P}_{(r)}.$ The standard flag is $E$: $=(\cdots\supset E_{r}\supset E_{r-1}\supset\cdots\supset E_{1}\supset\cdots),$ whose stabilizer in $\mathrm{G}$ is just the Borel subgroup $\mathrm{B}\subset \mathrm{P}_{(r)},$ and we have $\mathcal{B}^{r}\cong \mathrm{G}/\mathrm{B}.$

Let $\mathcal{F}_{\lambda}$ be the $\mathrm{G}$-orbit on $\mathcal{F}^{n}$ corresponding to $\lambda\in \Lambda_{\vartriangle}(n,r)=\mathfrak{S}_{r, n},$ then we have $\mathrm{G}/\mathrm{P}_{\lambda}\cong \mathcal{F}_{\lambda},$ and we have $\mathrm{P}_{\lambda}\supset \mathrm{B}$ with Weyl group $W_{\lambda}$: $=\{s_{i}\}_{i\in \hat{I}}\subseteq \mathfrak{S}_{r}.$

Recall that the orbits of $\mathrm{G}$ on $\mathcal{F}^{n}\times \mathcal{F}^{n}$ are indexed by the matrices $\mathfrak{S}_{r, n,n}=\Theta_{\vartriangle}(n,r).$ For $A\in \Theta_{\vartriangle}(n,r),$ we denote by $\mathcal{O}_{A}$ the corresponding orbit, which contains the representative $(\mathrm{P}_{\lambda}, w_{A}\mathrm{P}_{\mu})$ with $(\lambda,w_{A},\mu)\in \mathcal{B}$ and $j_{\vartriangle}(\lambda,w_{A},\mu)=A.$ Since the stabilizer of this representative $(\mathrm{P}_{\lambda}, w_{A}\mathrm{P}_{\mu})$ is $\mathrm{P}_{\lambda}\cap \mathrm{P}_{\mu}^{w_{A}},$ then we have $$\mathcal{O}_{A}\cong \mathrm{G}/(\mathrm{P}_{\lambda}\cap \mathrm{P}_{\mu}^{w_{A}})\cong \mathrm{G}\times^{\mathrm{P}_{\lambda}} \mathrm{P}_{\lambda}/(\mathrm{P}_{\lambda}\cap \mathrm{P}_{\mu}^{w_{A}}).$$

Let $pr_{1}: \mathcal{O}_{A}\rightarrow \mathcal{F}^{n}$ be the first projection on the first factor, then we have $pr_{1}(\mathcal{O}_{A})=\mathcal{F}_{\lambda}\cong \mathrm{G}/\mathrm{P}_{\lambda},$ and $\mathcal{O}_{A}\cong \mathrm{G}\times^{\mathrm{P}_{\lambda}} X_{A}^{L},$ where $L\in \mathcal{F}_{\lambda}.$ Consider the natural action of $\mathrm{P}_{\lambda}$ on $X_{A}^{L},$ it is easy to see that the stabilizer of a point $L'\in X_{A}^{L}$ is just $\mathrm{P}_{\lambda}\cap \mathrm{P}_{\mu}^{w_{A}}.$ Thus we have $\mathrm{P}_{\lambda}/(\mathrm{P}_{\lambda}\cap \mathrm{P}_{\mu}^{w_{A}})\cong X_{A}^{L}$ and $d_{A}$: $=$dim$(X_{A}^{L})=$dim$\mathrm{P}_{\lambda}/(\mathrm{P}_{\lambda}\cap \mathrm{P}_{\mu}^{w_{A}}).$

Let $w_{A}^{+}$ be the longest element in $W_{\lambda}w_{A}W_{\mu},$ then the following result gives another description of $d_{A}$ (see also [DF, Lemma 7.1]).
\begin{lemma}
Keep the above notation, we have $d_{A}= l(w_{A}^{+})-l(w_{0,\mu}).$
\end{lemma}
\begin{proof}
Note that $\mathrm{P}_{\lambda}w_{A}\mathrm{P}_{\mu}/\mathrm{P}_{\mu}$ is the $\mathrm{P}_{\lambda}$-orbit of $w_{A}\mathrm{P}_{\mu}/\mathrm{P}_{\mu}.$ Since we have Stab$_{\mathrm{P}_{\lambda}}(w_{A}\mathrm{P}_{\mu}/\mathrm{P}_{\mu})=\mathrm{P}_{\lambda}\cap \mathrm{P}_{\mu}^{w_{A}},$ thus we have as a variety:$$\mathrm{P}_{\lambda}w_{A}\mathrm{P}_{\mu}/\mathrm{P}_{\mu}\cong \mathrm{P}_{\lambda}/(\mathrm{P}_{\lambda}\cap \mathrm{P}_{\mu}^{w_{A}}).$$
Moreover, note that $$\mathrm{P}_{\lambda}w_{A}\mathrm{P}_{\mu}/\mathrm{B}=\bigsqcup_{y\in W_{\lambda}w_{A}W_{\mu}}\mathrm{B}y\mathrm{B}/\mathrm{B}.$$
From this we get dim$\mathrm{P}_{\lambda}w_{A}\mathrm{P}_{\mu}/\mathrm{B}=l(w_{A}^{+}).$ Now, we have\vskip2mm
$\hspace*{1.7cm}d_{A}=\mathrm{dim} \mathrm{P}_{\lambda}w_{A}\mathrm{P}_{\mu}/\mathrm{B}- \dim \mathrm{P}_{\mu}/\mathrm{B}=l(w_{A}^{+})-l(w_{0,\mu}).$\end{proof}

7.3 Now we recall the construction of Lusztig's canonical basis $\mathfrak{B}_{r}$ in $\mathfrak{U}_{r,n,n},$ which is consisting of elements $\{A\}, A\in \mathfrak{S}_{r, n,n}$ (see [L7, $\S4$]). Fix $A\in \mathfrak{S}_{r, n,n},$ and $L\in \mathcal{F}_{ro(A)}.$ The space $\mathcal{F}^{n}$ can be given the structure of an ind-scheme such that each set $X_{A}^{L}$ lies naturally in a projective algebraic variety. This follows from the fact that if we fix $i_{0}, j_{0}\in \mathbb{Z},$ then the subsets $$\mathcal{F}_{b, L}^{p}=\{L'\in \mathcal{F}_{b}|~\epsilon^{p}L_{i_{0}}\subset L_{j_{0}}'\subset \epsilon^{-p}L_{i_{0}}\}~~~(\mathrm{for}~p=1,2,\ldots)$$
are naturally projective algebraic varieties, each one embedded in the next, and for any fixed $A\in \mathfrak{S}_{r, n,n},$ there is a $p_{0}\in \mathbb{Z}$ such that $X_{A}^{L}$ is a locally closed subset of $\mathcal{F}_{b, L}^{p}$ for all $p\geq p_{0}.$ Thus its closure $\bar{X}_{A}^{L}$ is naturally a projective algebraic variety, which is independent of the choices of $i_{0}, j_{0}, p.$

For $A'\in \mathfrak{S}_{r, n,n},$ we write $A'\leq A$ if $X_{A'}^{L}\subseteq \bar{X}_{A}^{L}.$ Let $IC_{A}^{L}=IC(\bar{X}_{A}^{L})$ denote the simple perverse sheaf on $\bar{X}_{A}^{L}$ whose restriction to $X_{A}^{L}$ is $\mathbf{k},$ and let $(\mathcal{H}^{s}(IC_{A}^{L}))_{y}$ denote the stalk of $\mathcal{H}^{s}(IC_{A}^{L})$ at a point $y\in X_{A'}^{L},$ which is independent of the choice of $y.$ We define for $A'\leq A$ $$P_{A', A}=\sum\limits_{s}\dim (\mathcal{H}^{s}(IC_{A}^{L}))_{y} v^{s}\in \mathcal{A}.$$
Note that $P_{A', A}v^{-d_{A}+d_{A'}}=\Pi_{A', A},$ the polynomial defined in [L7, 4.1 (c)]. Following [L7, 4.1 (d)], we define a new basis for $\mathfrak{U}_{r,n,n}$ by $$\{A\}=\sum\limits_{A'; A'\leq A}\Pi_{A',A}[A'].$$
\begin{theorem}

Under the identification in Lemma 7.1, we have $\theta_{A}\cong \{A\}.$
\end{theorem}
\begin{proof}
From Lemma 7.2, we have $[A']=\widehat{\phi}_{A'}.$ Since $-d_{A}+d_{A'}=-l(w_{A}^{+})+l(w_{A'}^{+}),$ it suffices from (3.5) to prove $P_{A', A}=P_{w_{A'}^{+}, w_{A}^{+}}.$
To prove this, we consider the following natural projection: $\pi: \mathrm{G}/\mathrm{B}\rightarrow \mathcal{F}_{co(A)}.$ Let $\mathcal{O}_{w_{A}^{+}}$ be the $\mathrm{G}$-orbit in $\mathrm{G}/\mathrm{B}$ corresponding to $w_{A}^{+}\in W,$ and let $IC_{w_{A}^{+}}=IC(\bar{\mathcal{O}}_{w_{A}^{+}})$ be the intersection cohomology complex of $\bar{\mathcal{O}}_{w_{A}^{+}}$ whose restriction to $\mathcal{O}_{w_{A}^{+}}$ is the constant sheaf $\mathbf{k}.$ Then $\pi$ induces a map $\bar{\pi}: \bar{\mathcal{O}}_{w_{A}^{+}}\rightarrow \bar{X}_{A}^{L}.$ By [GM, 5.4.2] we have $\bar{\pi}^{*}(IC_{A}^{L})=IC_{w_{A}^{+}}.$ Since $\bar{\pi}$ is a fiber bundle bundle with smooth fibers (isomorphic to a suitable partial flag variety), it follows that $$(\mathcal{H}^{s}(IC_{w_{A}^{+}}))_{y}\cong (\bar{\pi}^{*}\mathcal{H}^{s}(IC_{A}^{L})_{y}\cong (\mathcal{H}^{s}(IC_{A}^{L}))_{\bar{\pi}(y)},$$
where $(\mathcal{H}^{s}(IC_{w_{A}^{+}}))_{y}$ denotes the stalk at any point in $\mathcal{O}_{y}.$ Now, from [KL2, $\S5$] we have
$$P_{w_{A'}^{+}, w_{A}^{+}}=\sum\limits_{s}\dim (\mathcal{H}^{s}(IC_{w_{A}^{+}}))_{w_{A'}^{+}} v^{s}.$$
Consequently, $P_{w_{A'}^{+}, w_{A}^{+}}=P_{A', A}$ as desired.
\end{proof}

In [Gr, Question 4.3.4], He has raised the question whether the structure constants for $\mathcal{S}_{q}^{\vartriangle}(n,r)$ with respect to the K-L basis $\{\theta_{A}\}_{A\in \Theta_{\vartriangle}(n,r)}$ lie in $\mathbb{N}[v,v^{-1}].$ Since the structure constants for $\mathfrak{U}_{r,n,n}$ with respect to the canonical basis $\{A\},$ $A\in \mathfrak{S}_{r, n,n}$ lie in $\mathbb{N}[v,v^{-1}]$ by [L7, $\S4.2$ and $\S4.5$], thus, we get that the structure constants for the K-L basis $\{\theta_{A}\}_{A\in \Theta_{\vartriangle}(n,r)}$ also lie in $\mathbb{N}[v,v^{-1}]$ by Lemma 7.1 and Theorem 7.1.



Mathematical Sciences Center, Tsinghua University, Jin Chun Yuan West Building.

Beijing, 100084, P. R. China.

E-mail address: cwdeng@amss.ac.cn

\end{document}